\documentclass[10pt]{article} 
\usepackage{amsmath,amssymb,amsthm,amsfonts,amscd,euscript,verbatim, t1enc, newlfont, xypic, graphicx} 
\usepackage{hyperref}
\usepackage{cancel} 
\providecommand{\keywords}[1]{\textbf{\textit{Key words and phrases }} #1}
\providecommand{\subjclass}[1]{\textbf{\textit{2010 Mathematics Subject Classification.}} #1}
\hypersetup{breaklinks}
\theoremstyle{definition}
\newtheorem{theo}{Theorem}[subsection]

\newtheorem{pr}[theo]{Proposition}

 \newtheorem{lem}[theo]{Lemma}
 \newtheorem{coro}[theo]{Corollary}
  
\theoremstyle{remark}
\newtheorem{rema}[theo]{Remark}

\newtheorem{defi}[theo]{Definition}

 \newcommand\lan{\langle}
\newcommand\ra{\rangle}
\newcommand\ral{\rangle^{\al}}
\newcommand\ralo{\rangle^{\aleph_1}}
\newcommand\rab{\rangle^{\operatorname{cl}}}
\newcommand\bral{]^{\al}}
\newcommand\brab{]^{\operatorname{cl}}}
\newcommand\bralo{]^{\aleph_1}}

\newcommand\ob{^{-1}}

\newcommand\obj{\operatorname{Obj}}
\newcommand\mo{\operatorname{Mor}}
\newcommand\id{\operatorname{id}}

\newcommand\hu{\underline{H}}
\newcommand\cu{\underline{C}}
\newcommand\du{\underline{D}}
\newcommand\au{\underline{A}}
\newcommand\bu{\underline{B}}
\newcommand\eu{\underline{E}}

\newcommand\abo{{_{a-}}}

\newcommand\rinf{R_{<\infty}}
\newcommand\linf{L_{+\infty}}

\newcommand\n{\mathbb{N}}
\newcommand\z{{\mathbb{Z}}}

\newcommand\ab{\underline{\operatorname{Ab}}}

\newcommand\al{\alpha}
\newcommand\be{\beta}

\newcommand\ffab{\underline{\operatorname{FGFrAb}}}
\newcommand\shtop{\operatorname{SH}}
\newcommand\sph{_{\operatorname{sph}}}
\newcommand\fin{\operatorname{fin}}
\newcommand\shfin{\shtop^{\fin}}
\newcommand\ns{\{0\}}

\DeclareMathOperator\prli{\varprojlim}
\DeclareMathOperator\inli{\varinjlim}
\DeclareMathOperator\hcl{\underrightarrow{\operatorname{holim}}}

\DeclareMathOperator\co{\operatorname{Cone}}

\DeclareMathOperator\kar{\operatorname{Kar}}

\newcommand\hrt{{\underline{Ht}}}

\newcommand\hw{{\underline{Hw}}}
\newcommand\hv{{\underline{Hv}}}
\newcommand\hf{HF}
\newcommand\hpi{H\pi} 

\newcommand\cuw{\cu_w}

\newcommand\wstu{w^{st}} 

\newcommand\perpp{{}^{\perp}}

\newcommand\lo{\mathcal{LO}}

\newcommand\cp{\mathcal{P}}
\newcommand\cq{\mathcal{Q}}

\newcommand\alz
{{\aleph_0}}
\newcommand\alo{{\aleph_1}}

\newcommand\w{{\mathfrak{w}}}

\newcommand\tal{{^{\al}}}
\newcommand\otal{{^{\oplus \al}}}
\newcommand\oi{{^{\oplus I}}}
\newcommand\tbig{{^{\operatorname{cl}}}}

\numberwithin{equation}{subsection}

\begin{document}

 \title{On purely generated $\al$-smashing weight structures and weight-exact localizations} 
 \author{Mikhail V. Bondarko, Vladimir A. Sosnilo
   \thanks{Research is supported by the Russian Science Foundation grant no. 16-11-10073.}}\maketitle
\begin{abstract}
This paper is dedicated to new methods of constructing weight structures and weight-exact localizations; our arguments generalize their bounded versions considered in previous papers of the authors. We start from a class of objects $\cp$ of a triangulated category $\cu$ that satisfies a certain {\it (countable) negativity} condition (there are no $\cu$-extensions of positive degrees between elements of $\cp$; we actually need a somewhat stronger condition of this sort) to obtain a weight structure both "halves" of which are closed either with respect to $\cu$-coproducts of less than $\al$ objects (where $\al$ is a fixed regular cardinal) or with respect to all coproducts (provided that $\cu$ is closed with respect to coproducts of this sort). This construction gives all "reasonable" weight structures satisfying the latter conditions. In particular, 
 one can 
 obtain certain weight structures on spectra (in $\shtop$) consisting of less than $\al$ cells, 
 and on certain localizations of $\shtop$; these results are new.  

\end{abstract}



\tableofcontents

 \section*{Introduction} 

This paper is dedicated to new methods of constructing weight structures and weight-exact localizations; we mostly consider triangulated categories that are closed with respect to coproducts (at least, of bounded cardinality). Our results vastly extend the properties of bounded weight structures considered in previous papers.

 We recall that weight structures  are important counterparts of $t$-structures; they were independently introduced by the first author (in \cite{bws}) and  by D. Pauksztello (under the name of co-$t$-structures; see \cite{paucomp}).  Weight structures have found several applications to motives and representation theory, and there also exist Hodge-theoretic and "topological" examples (we will often mention the latter in the current paper). Somewhat similarly to $t$-structures, a weight structure $w$ on a triangulated category $\cu$ is given by a couple of classes of objects $\cu_{w\ge 0}$ and $\cu_{w\le 0}$ satisfying certain axioms.

Certainly, this makes  the construction and the classification of weight structures important problems. There are two main methods for constructing a weight structure on $\cu$ starting from a specified class of its objects. 

Firstly (as proved in \cite{bws}), one can  start from a class $N$ of objects that are "pure", i.e., they "should belong" to the heart  $\cu_{w= 0} =\cu_{w\ge 0}\cap\cu_{w\le 0}$ of some $w$. Then $N$ should be {\it negative} (i.e., there exist only zero morphisms between $N$ and $N[i]$ for $i>0$); if this is the case then there exists a unique {\it bounded} weight structure $w'$ on the smallest strict triangulated subcategory $\cu'$  of $\cu$ containing $N$ such that $N\subset \cu'_{w'= 0}$. 
  Moreover, this method yields all bounded weight structures. Yet it gives no unbounded weight structures; 
	 thus it is not appropriate for "large" triangulated categories (say, for those that are closed with respect to coproducts).\footnote{Note also that this method along with its generalization considered in the current paper cannot have any reasonable analogues for $t$-structures.} 

Secondly, for a given set $L$ one can look for a minimal class $\lo\supset L$ that can be completed to a weight structure (i.e., there exists $w$ such that $\lo=\cu_{w\le 0}$; one says that $w$ is {\it generated} by $L$ if $\cu_{w\ge 0}$ consists precisely of those 
$P\in \obj \cu$ such that there are only zero morphisms between $L$ and $P[i]$ for $i>0$). D. Pauksztello has proved in \cite{paucomp} that any set of {\it compact} objects of $\cu$ does generate some weight structure (note that this weight structure is automatically unique), and in \cite{bpure} it was proved that it suffices to assume that $L$ is a {\it perfect} set 
(as defined by A. Neeman and H. Krause).  
 Any weight structure obtained this way is automatically   {\it smashing}, i.e.,  $\cu_{w\ge 0}$ is closed with respect to small $\cu$-coproducts (whereas $\cu_{w\le 0}$ is closed with respect to coproducts automatically). 
This assumption on weight structures is rather reasonable since it gives certain "control of weights" (see Proposition 2.5.1 and Theorem 
 4.4.3 of \cite{bpure}, and Proposition 2.3.2 of \cite{bwcp}).\footnote{Moreover, for  any triangulated category $\cu$ that is {\it well generated} in the sense of Neeman,   any smashing weight structure $w$ on $\cu$ is generated by a perfect set of objects; furthermore, $w$ is also {\it (strongly) well generated} in a certain sense (see Theorem 4.4.3(II.2) of \cite{bpure}). The latter statement doesn't make the construction of perfectly generated weight structures that are not compactly generated easy; however, a rich family of examples  
 is mentioned in Remark 4.3.10 of ibid. (these weight structures are dual to the ones constructed in Corollary 4.3.9(2) of ibid. by means of Brown-Comenetz duality).} 

However, the second approach does not 
 give a "precise" description of the weight structure generated by 
  $L$. Moreover, one surely cannot apply it (in its current form) to the classification of weight structures that are just {\it $\al$-smashing} (i.e., $\cu$ and  $\cu_{w\le 0}$ are closed only with respect to coproducts of less than $\al$ objects, where $\al$ is a fixed infinite regular cardinal number). 
	
	In the current paper  we  generalize the first method to obtain a very general construction of $\al$-smashing weight structures (starting from a class $N$ that satisfies a certain stronger version of the negativity condition). 
	Our  Theorem \ref{tvneg} essentially (cf. Corollary \ref{classgws} where we extend this theorem to the setting of smashing weight structures) vastly generalizes 
	 \cite[Theorem 4.5.2]{bws};\footnote{And the argument we use for the proofs is quite different from the clumsy argument used for the proof of loc. cit.} 
thus	it yields the so-called {\it spherical} weight structure $w\sph$ on the stable homotopy category (see \S4.6 of ibid., \S4.2 of \cite{bwcp}, and \S4.2 of \cite{bkwn}) and gives some new information on it. 
 Furthermore, for any $\al$-smashing or smashing weight structure on $\cu$ our new method gives a "compatible" weight structure on a "rather large" subcategory of $\cu$. Under certain quite reasonable conditions  on $(\cu,w)$ this subcategory equals  $\cu$ (see 
Corollaries \ref{classgws}(3) and 
 \ref{cabo}(II.2) below); thus 
 the new approach yields all weight structures satisfying these conditions. Taking into account (also) other parts of Theorem \ref{tvneg} and  Corollary \ref{classgws} one may say that these results give the tools to "deal with $\al$-generated (and {\it class-generated}) weight structures as if they were bounded".

A related question (that yields plenty of non-trivial examples for 
 our existence of weight structures results) is when on a Verdier  localization $\cu'$ of $\cu$ by a (triangulated) subcategory $\du$ there exists a weight structure $w'$ such that the localization functor $\pi$ sends $\cu_{w\le 0}$ into  $\cu'_{w'\le 0}$ and  $\cu_{w\ge 0}$ into  $\cu'_{w'\ge 0}$ (one says that $\pi$ is {\it weight-exact} and $w$ {\it descends} to $\cu'$).  
In \cite{bososn} we have constructed a vast family of examples of this situation (note that the existence of $w'$ of this sort 
is a much weaker assumption than the one that $w$ "restricts" to $\du$; this is one of the settings where weight structures and $t$-structures behave quite differently). In the current paper we  give a certain if and only if criterion for the existence of $w'$ on $\cu'$, 
 and also prove that if $w'$ 
 exists  then it can be easily and canonically described in terms of $w$. We  also describe a general construction that produces  a vast family of weight-exact localizations and computes their hearts;  
this result generalizes the 
(main general) Theorems 4.2.3 and 4.3.1.4  of \cite{bososn}. Moreover, we 
 obtain some new results on localizations of additive categories.


Let us  now describe the contents  of the paper. Some more information of this sort may be found in the beginnings of sections. 

Section \ref{sprelim} is mostly dedicated to recollections on categories and weight structures; yet the easy Proposition \ref{prestr} is new.

In \S\ref{sgenw} we prove some nice general lemmas on triangulated categories and use them to prove our main "construction of weight structures" Theorem \ref{tvneg} (yet one of its parts and some related results are proved in \S\ref{salmb} only; see also Corollary \ref{classgws} for the "smashing" version of this theorem).  
We also discuss examples for our theorem (this includes the 
 spherical weight structure on $\shtop$), and the relation of our construction to earlier results in this direction (i.e., to compactly generated and perfectly generated weight structures).

In \S\ref{sloc} we study weight-exact localizations and their hearts (similarly to the aforementioned results of \cite{bososn}, the heart $\hw'$ of $w'$ can 
 be described in terms of $\hw$ only). 
We also discuss the relation of our results to so-called generalized universal localizations of rings (other related terms are {\it silting} and {\it tilting} objects; cf. Definition 5.1 of \cite{bapa}). 

In \S\ref{suppl} we use the theory of {\it countable homotopy colimits} (in triangulated categories) to construct   and study certain  functors $\rinf$ and $\linf$ adjoint to embeddings, and to prove that the smallest triangulated subcategory of $\cu$ 
that contains $\cu_{w=0}$ and is closed under countable coproducts also contains $\cu_{w\ge 0}$ if $w$ is {\it left non-degenerate}; this subcategory equals $\cu$ if we assume some more (rather reasonable) conditions. We also give some examples and comments to the results of \S\ref{sloc}.
Next we recall some basics of the theory of weight complexes; it is used for the study of hearts in the previous section and also can be used for the calculation of the intersection of triangulated $\cu_1,\cu_2\subset \cu$ (we formulate a statement of this sort that significantly generalizes the main abstract result of \cite{binters}). 

The authors are deeply grateful to the referee for his really helpful remarks.

\section{Preliminaries}\label{sprelim} 

In \S\ref{snotata} we introduce some 
 definitions; possibly, the most "unusual" of them is the notion of an $\al$-smashing category.

In \S\ref{sbw} we recall some basics on weight structures and prove a few new results (that are rather easy). 

\subsection{Some  definitions and conventions}\label{snotata} 

We introduce some notions that will be used throughout the paper. Most of our definitions and conventions are  well-known.
\begin{itemize}


\item Given a category $C$ and  $X,Y\in\obj C$  we will write
$C(X,Y)$ for  the set of morphisms from $X$ to $Y$ in $C$.

\item For categories $C',C$ we write $C'\subset C$ if $C'$ is a full 
subcategory of $C$.

\item Given a category $C$ and  $X,Y\in\obj C$, we say that $X$ is a {\it
retract} of $Y$ 
 if $\id_X$ can be 
 factored through $Y$.\footnote{Certainly,  if $C$ is triangulated or abelian,
then $X$ is a retract of $Y$ if and only if $X$ is its direct summand.}\ 

\item An additive subcategory $\bu$ of additive category $C$ 
is said to be {\it Karoubi-closed}
  in $C$ if it contains all $C$-retracts of its objects.
The full subcategory $\kar_{C}(\bu)$ of  $C$ whose objects
are all $C$-retracts of objects of $\bu$ will be called the {\it Karoubi-closure} of $\bu$ in $C$. 

\item The {\it Karoubi envelope} $\kar(\bu)$ (no lower index) of an additive
category $\bu$ is the category of ``formal images'' of idempotents in $\bu$.
So, its objects are the pairs $(A,p)$ for $A\in \obj \bu,\ p\in \bu(A,A),\ p^2=p$, and the morphisms are given by the formula 
$$\kar(\bu)((X,p),(X',p'))=\{f\in \bu(X,X'):\ p'\circ f=f \circ p=f \}.$$ 
 The correspondence  $A\mapsto (A,\id_A)$ (for $A\in \obj \bu$) fully embeds $\bu$ into $\kar(\bu)$.
 Moreover, $\kar(\bu)$ is {\it Karoubian}, i.e., 
 any idempotent morphism yields a direct sum decomposition in 
 $\kar(\bu)$. 


\item For an additive category $\bu$ we  write $K(\bu)$ for the homotopy category of (cohomological) complexes over $\bu$; its full subcategory of bounded complexes will be denoted by $K^b(B)$. 
We will write $M=(M^i)$ if $M^i\in \obj \bu$ are the terms of the complex $M$.

\item A cardinal $\al$ is said to be {\it regular} if it cannot be presented as a sum of less then $\al$ cardinals that are less than $\al$.

Throughout the paper $\al$ will denote some infinite regular cardinal 
that one can usually assume to be fixed. 

\item We will say that an additive category $\bu$ is {\it $\al$-smashing} whenever it contains coproducts for arbitrary  families of its objects of cardinality less than $\al$.
For $\bu$ satisfying this condition we will say that a class $\cp\subset \obj \bu$ is {\it $\al$-smashing} (in $\bu$) if it is closed with respect to  $\bu$-coproducts of less than $\al$ of its elements.\footnote{Note that any triangulated category is $\alz$-smashing, and an $\alz$-smashing subclass is just an additive one. 
  Moreover, one can certainly consider $\al$-smashing classes inside a $\be$-smashing category $\bu$ whenever $\be\ge \al$.}
Respectively, we will say that  some $\cp\subset \obj \bu$ or $\bu$ itself   is {\it smashing} whenever the corresponding object class 
 contains $\bu$-coproducts of arbitrary small families of its elements.\footnote{Here and below 
the "smashing versions" of our definitions and statements 
 are essentially particular cases of the $\al$-smashing ones; we will say more on this matter soon.}
So, if we will say that some $\cp\subset \obj \bu$ is $\al$-smashing (resp., smashing) then we will automatically mean that $\bu$ is $\al$-smashing (resp., smashing) as well. 

\item The symbol $\cu$ below will always denote some triangulated category;
usually it will be endowed with a weight structure $w$. 
The  symbols $\cu'$ and $\du$ will also be used for triangulated categories only.

\item For $M,N\in \obj \cu$ we will write $M\perp N$ if $\cu(M,N)=\ns$. For
$X,Y\subset \obj \cu$ we write $X\perp Y$ if $M\perp N$ for all $M\in X,\ N\in Y$.
Given $\cp\subset\obj \cu$ we  will write $\cp^\perp$ for the class $$\{N\in \obj \cu:\ M\perp N\ \forall M\in \cp\}.$$
Dually, $\perpp \cp=\{M\in \obj \cu:\ M\perp N\ \forall N\in \cp\}$.

\item Given $f\in\cu (M,N)$, where $M,N\in\obj\cu$, we will call the third vertex of (any) distinguished triangle $M\stackrel{f}{\to}N\to P$ a {\it cone} of
$f$.\footnote{Recall that different choices of cones are connected by non-unique isomorphisms.}\


\item Assume that $\cp$ is a subclass of $\obj \cu$. We will say that $\cp$ is {\it strict} if it  contains all objects of $\cu$ isomorphic to its elements.

\item For any  $A,B,C \in \obj\cu$ we will say that $C$ is an {\it extension} of $B$ by $A$ if there exists a distinguished triangle $A \to C \to B \to A[1]$.

\item
$\cp\subset \obj \cu$ is said to be  {\it extension-closed}
    if it 
		is closed with respect to extensions and contains $0$ (hence it is also strict). 
		
		We will call the smallest extension-closed subclass 
of objects of $\cu$ that  contains a given class $B\subset \obj\cu$ 
  the {\it extension-closure} of $B$. 

\item We will say that  
 $\cp$ {\it strongly generates} a subcategory $\du\subset \cu$ 
 if $\du$ is the smallest full strict triangulated subcategory of $\cu$ such that $\cp\subset \obj \du$.\footnote{Certainly, 
this condition is equivalent to $ \du$ being the extension-closure of $\cup_{j\in \z}\cp[j]$.}

\item If $\cu$ is an $\al$-smashing (resp., smashing)  triangulated category and  $\cp\subset \obj \du$ then we will write $ [\cp\bral$ (resp. $ [\cp\brab$) 
for the smallest extension-closed $\al$-smashing (resp., smashing) subclass of $\obj \cu$ that contains $\cp$.\footnote{The notation $ [\cp\brab$ is "justified" in Definition \ref{dwso}(\ref{idcgen},\ref{id6}) below.  Note also that  $ [\cp]^{\alz}$ is just the extension-closure of $\cp$.} 

Moreover, we will say that a full triangulated subcategory $\du\subset \cu$ is an {\it $\al$-localizing} (resp., {\it localizing})  whenever the class $\obj \du$ is $\al$-smashing (resp., smashing) in $\cu$.
Respectively, we will call the full subcategory of $\cu$ whose objects class equals $ [(\cup_{i\in \z}\cp[i])\bral$ (resp. $[(\cup_{i\in \z}\cp[i])\brab$)  the {\it $\al$-localizing (resp., localizing) subcategory of $\cu$ generated by $\cp$} (note that this subcategory is certainly triangulated); we will write $\lan \cp\ral$ (resp. $\lan \cp \rab$) for this subcategory.

\end{itemize}

\begin{rema}\label{rclass}
1. In this paper  all categories are locally small, i.e., morphism classes (between two objects) in all categories  are assumed to be sets. We will use the term {\it class-categories} for those "generalized" categories that do not necessarily satisfy this condition.

2. Certainly, class-categories become locally small and smashing categories become $\al$-smashing for the corresponding $\al$ if one passes to a larger universe. We 
prefer to avoid this point of view below; 
 this is usually easy (yet cf. Remark \ref{rheart}(\ref{irhcl}) below). However, the reader may easily note that some of our 
 proofs can be simplified (in the obvious way) using larger 
 universes. On the other hand, the reader interested in the arguments overcoming set-theoretic difficulties certainly should mind the distinction between the words "set" and "class" in the text. 


3. Since the only class-categories in this paper that are not necessarily categories are 
  localizations of (locally small) categories, we 
	 we will not mention class-categories in our formulations. For this reason we will sometimes say that a Verdier quotient category {\it exists} 
 to mean that it is a (locally small) category. 

Yet these restrictions can be ignored. 
In particular, 
 this is the case for part 2 of the following well-known  proposition. 
\end{rema}

\begin{pr}\label{pneem}
Assume that $\cu$ is an $\al$-smashing category, where $\al$ is an (infinite) regular cardinal.

1. Then the coproduct of any family of $\cu$-distinguished triangles that is of cardinality less than $\al$ is distinguished also.

2. Let $\du$ be an $\al$-localizing subcategory of $\cu$, and assume that the Verdier localization $\cu'=\cu/\du$ exists. Then $\cu'$ is $\al$-smashing also and the localization functor $\pi$ commutes with coproducts of less than $\al$ objects. 
\end{pr}
\begin{proof}
1. Immediate from Remark 1.2.2 of \cite{neebook}. 

2. This is Lemma 3.2.10 of ibid.
\end{proof}

\subsection{Weight structures: basics}\label{sbw}

\begin{defi}\label{dwstr}

 A couple of subclasses $\cu_{w\le 0}$ and $ \cu_{w\ge 0}\subset\obj \cu$ 
will be said to define a {\it weight structure} $w$ on a triangulated category  $\cu$ if 
they  satisfy the following conditions.

(i) $\cu_{w\le 0}$ and $\cu_{w\ge 0}$ are 
Karoubi-closed in $\cu$ (i.e., contain all $\cu$-retracts of their elements).

(ii) {\bf Semi-invariance with respect to translations.}

$\cu_{w\le 0}\subset \cu_{w\le 0}[1]$ and $\cu_{w\ge 0}[1]\subset
\cu_{w\ge 0}$.

(iii) {\bf Orthogonality.}

$\cu_{w\le 0}\perp \cu_{w\ge 0}[1]$.

(iv) {\bf Weight decompositions}.

 For any $M\in\obj \cu$ there
exists a distinguished triangle
$LM\to M\to RM
{\to} LM[1]$
such that $LM\in \cu_{w\le 0} $ and $ RM\in \cu_{w\ge 0}[1]$.
\end{defi}

We will also need the following definitions.

\begin{defi}\label{dwso} 

Let $i,j\in \z$; assume that $\cu$ is a triangulated category endowed with a weight structure $w$.

\begin{enumerate}
\item\label{id1} The full subcategory  $\hw\subset \cu$ whose object class is $\cu_{w=0}=\cu_{w\ge 0}\cap \cu_{w\le 0}$ 
 is called the {\it heart} of  $w$.

\item\label{id2} $\cu_{w\ge i}$ (resp. $\cu_{w\le i}$,  $\cu_{w= i}$) will denote $\cu_{w\ge 0}[i]$ (resp. $\cu_{w\le 0}[i]$,  $\cu_{w= 0}[i]$).

\item\label{id3} $\cu_{[i,j]}$  denotes $\cu_{w\ge i}\cap \cu_{w\le j}$; so, this class  equals $\ns$ if $i>j$.

$\cu_b\subset \cu$ (resp.   $\cu_-$, resp. $\cu_+)$  will be the full subcategory  of $\cu$ whose object class is $\cup_{i,j\in \z}\cu_{[i,j]}$ (resp. $\cup_{i\in \z} \cu_{w\le i}$, resp. $\cup_{i\in \z}\cu_{w\ge i}$). We will say that the objects of these subcategories are $w$-bounded, $w$-bounded above, and $w$-bounded below, respectively.

\item\label{idndeg} We will say that $M\in \obj \cu$ is 
{\it right (resp. left) $w$-degenerate} if  it belongs to $\cap_{l\in z} \cu_{w\le l}$ (resp. to $\cap_{l\in \z} \cu_{w\ge l}$).

Respectively, we will say that $w$ is {\it right (resp. left) non-degenerate} whenever all its right (resp. left) degenerate objects are zero. 

\item\label{id5} Let $\cu'$  be a triangulated category endowed with a weight structure  $w'$; let $F:\cu\to \cu'$ be an exact functor.

We will say that $F$ is {\it left weight-exact} (with respect to $w,w'$) if it maps $\cu_{w\le 0}$ into $\cu'_{w'\le 0}$; it will be called {\it right weight-exact} if it sends $\cu_{w\ge 0}$ into $\cu'_{w'\ge 0}$. $F$ is called {\it weight-exact} if it is both left and right weight-exact. In the latter case $F$ obviously restricts to an additive functor $\hw\to \hw'$ that  will be denoted by $\hf$.


\item\label{idrestr} For a triangulated subcategory $\du\subset \cu$ we will say that $w$ {\it restricts} to $\du$ whenever the couple $(\cu_{w\le 0}\cap \obj \du,\cu_{w\ge 0}\cap \obj \du)$ 
 is a weight structure  on $\du$.

\item\label{idsmash} For a regular cardinal $\al$ we will say that $w$ is {\it $\al$-smashing (resp., smashing)} if $\cu$ is $\al$-smashing (resp., smashing)  and the class $\cu_{w\ge 0}$ is so as well (and we will not usually say explicitly that $\cu$ is $\al$-smashing or smashing, respectively). 

Moreover, we will 
use the term "{\it countably smashing}" 
as a synonym for "$\alo$-smashing".

\item\label{idcgen} We will say that a class $\cp\subset \obj \cu$ {\it $\al$-generates} (resp., class-generates)  $w$ whenever $\cu$ 
 is $\al$-smashing, $ \cu_{w\ge 0}=[\cup_{i\ge 0}\cp[i]\bral$, and $ \cu_{w\le 0}=[\cup_{i\le 0}\cp[i]\bral$ (resp., $\cu$ is smashing,  $ \cu_{w\ge 0}=[\cup_{i\ge 0}\cp[i]\brab$, and $ \cu_{w\le 0}=[\cup_{i\le 0}\cp[i]\brab$; see \S\ref{snotata}).

\item\label{igen} We will say that  $\cp\subset \obj \cu$ {\it generates} $w$  if $\cu_{w\ge 0}=(\cup_{i<0}\cp[i])\perpp$.

\item\label{id6} Let $\cp$ be a 
class of objects of an $\al$-smashing (resp., smashing) triangulated category $\du$.

We will say that $\cp$ is {\it $\al$-negative} (resp., {\it class-negative}) (in $\du$) if  $\cp\perp_{\du} [\cup_{i> 0}\cp[i]\bral_{\du}$ (resp. $\cp\perp_{\du} [\cup_{i> 0}\cp[i]\brab_{\du}$). 
$\alo$-negative classes of objects will also be called {\it countably negative} ones.\footnote{Theorem \ref{tvneg}(6) below demonstrates that this case of $\al$-negativity is of special importance. } 


\end{enumerate}
\end{defi}

\begin{rema}\label{rstws} 

1. A  simple (though rather important) example of a weight structure comes from the stupid
filtration on $K(\bu)$ (or on $K^b(\bu)$)   for an arbitrary additive category $\bu$.

In either of these  categories we take
$\cu_{w\le 0}$ (resp. $\cu_{w\ge 0}$) to be the class of objects in $\cu$ 
 that are homotopy equivalent to those complexes in $\cu\subset K(\bu)$ that are concentrated in degrees $\ge 0$ (resp. $\le 0$).  
Then one can take weight decompositions of objects  given by stupid filtrations of complexes (cf.  Remark 1.2.3(1) of \cite{bonspkar}); so we will use the notation $\wstu$ for  this weight structure.

 The heart of this {\it stupid} weight structure 
is the Karoubi-closure  of $\bu$ in $\cu$; thus for $K(\bu)$ this heart is equivalent to $\kar(\bu)$ (see Remark \ref{rkar}  below). 

2. A weight decomposition of a fixed object $M$ of $\cu$ is (almost) never canonical. 

Still  some choice of a weight decomposition of $M[-m]$ shifted by $[m]$ (for some $m\in \z$) is often needed. 
So we take a distinguished triangle \begin{equation}\label{ewd} w_{\le m}M\to M\to w_{\ge m+1}M\to (w_{\le m}M)[1] \end{equation} 
with some $ w_{\ge m+1}M\in \cu_{w\ge m+1}$, $ w_{\le m}M\in \cu_{w\le m}$; we will call it an {\it $m$-weight decomposition} of $M$. 

 We will 
  use this notation below (even though $w_{\ge m+1}M$ and $ w_{\le m}M$ are not canonically determined by $M$). Moreover, when we will write arrows of the type $w_{\le m}M\to M$ or $M\to w_{\ge m+1}M$ we will always assume that they come from some $m$-weight decomposition of $M$. 
 
3. In the current paper we use the ``homological convention'' for weight structures; it was previously used in \cite{brelmot}, \cite{bpure}, \cite{bokum}, \cite{bonspkar},  \cite{bsosnl}, \cite{binters}, \cite{bvt},  \cite{bgn},  \cite{bwcp}, \cite{bkwn}, and \cite{bvtr},   
  whereas in 
\cite{bws} 
the roles of $\cu_{w\le 0}$ and $\cu_{w\ge 0}$ are interchanged, i.e., one
considers   $\cu^{w\le 0}=\cu_{w\ge 0}$ and $\cu^{w\ge 0}=\cu_{w\le 0}$. 

4. Moreover, in \cite{bws} both "halves" of $w$ were required to be  additive. Yet this additional restriction is easily seen to follow from the remaining axioms; see Remark 1.2.3(4) of \cite{bonspkar} (or Proposition \ref{pbw}(\ref{iext}) below).
 
Thus any weight structure is $\alz$-smashing. 
\end{rema}

Now we list a collection of properties of weight structures; 
we assume that $\cu$ is (a triangulated category) endowed with a (fixed) weight structure $w$.

\begin{pr} \label{pbw}
Let  
 $M,M'\in \obj \cu$, $m\le n \in \z$. 

\begin{enumerate}
\item \label{idual} 
The axiomatics of weight structures is self-dual, i.e., on the category $\du=\cu^{op}$
(so, $\obj\du=\obj\cu$) there exists the (opposite)  weight
structure $w^{op}$ for which $\du_{w^{op}\le 0}=\cu_{w\ge 0}$ and
$\du_{w^{op}\ge 0}=\cu_{w\le 0}$.

\item\label{iextb} $\cu_b$,  $\cu_+$, and  $\cu_-$  are   Karoubi-closed triangulated subcategories of $\cu$,  $w$ restricts to each of them, and the hearts of the resulting weight structures equal $\hw$.

 \item\label{iort}
 $\cu_{w\ge 0}=(\cu_{w\le -1})^{\perp}$ and $\cu_{w\le 0}={}^{\perp} \cu_{w\ge 1}$.

\item\label{iext}  $\cu_{w\le 0}$, $\cu_{w\ge 0}$, and $\cu_{w=0}$ are (additive and) extension-closed. 

\item\label{icopr} $\cu_{w\le 0}$ is closed with respect to $\cu$-coproducts (i.e., it contains all those small coproducts of its elements that exist in $\cu$).

\item\label{isplit} If $A\to B\to C\to A[1]$ is a $\cu$-distinguished triangle 
 and $A,C\in  \cu_{w=0}$ then this distinguished triangle splits; hence $B\cong A\bigoplus C\in \cu_{w=0}$.

\item\label{icompl} 
 For any (fixed) $m$-weight decomposition of $M$ and an $n$-weight decomposition of $M'$  (see Remark \ref{rstws}(2)) 
any  morphism $g\in \cu(M,M')$ 
can be extended 
to a 
morphism of the corresponding distinguished triangles:
 \begin{equation}\label{ecompl} \begin{CD} w_{\le m} M@>{
}>>
M@>{}>> w_{\ge m+1}M\\
@VV{h}V@VV{g}V@ VV{j}V \\
w_{\le n} M'@>{}>>
M'@>{}>> w_{\ge n+1}M' \end{CD}
\end{equation}

Moreover, if $m<n$ then this extension is unique (provided that the rows are fixed).

\item\label{ifilt} Take arbitrary choices of  $w_{\le i}M$ for all $i\in \z$. Then there exist  unique morphisms 
 $c_i:w_{\le i-1}M\to  w_{\le i}M$ (for all $i\in \z$) that are "compatible" with the weight decomposition morphisms $w_{\le j} M\to M$, and $\co(c_i)\in \cu_{w=i}$. 

\item\label{iextcub} 
 The class $\cu_{[m,n]}$ equals the extension-closure of  $\cup_{m\le i\le n}\cu_{w=i}$; thus $\obj \cu_b$ is the extension-closure of $\cup_{i\in\z}\cu_{w=i}$.

\item\label{i01} The class $\cu_{[0,1]}$ consists exactly of cones (in $\cu$) of morphisms in $\hw\subset \cu$. 

\item\label{iuni} Let  $v$ be another weight structure for $\cu$; assume   $\cu_{w\le 0}\subset \cu_{v\le 0}$ and $\cu_{w\ge 0}\subset \cu_{v\ge 0}$.      
  Then $w=v$ (i.e., these inclusions are equalities).

\item\label{iwd0}  If $M\in \cu_{w\ge m}$  then any choice of $w_{\le n}M$ belongs to $\cu_{[m,n]}$. 
Dually, if  $M\in \cu_{w\le n}$  then $w_{\ge m}M\in \cu_{[m,n]}$. 

\item\label{ifactp} If $M' \in \cu_{w\ge m}$ then any $g\in \cu(M,M')$ factors through (any possible choice of) $w_{\ge m}M$. Dually, if $M \in \cu_{w\le m}$ then $g$ factors through $w_{\le m}M'$.

\item\label{iwdmod} If $M$ belongs to $ \cu_{w\le 0}$ (resp. to $\cu_{w\ge 0}$) then it is a retract of any choice of $w_{\le 0}M$ (resp. of $w_{\ge 0}M$).

\item\label{igenw} $w$ is completely determined by any class $\cp$ that generates it, and we have $\cp\subset \cu_{w\le 0}$ for any $\cp$ of this sort.

\item\label{icoprt} If $\cu$ is $\al$-smashing then the class $\cu_{w\le 0}$ contains $[\cup_{i\le 0}\cu_{w=i}\bral$. 

\item\label{ismash} If $w$ is $\al$-smashing (and so, $\cu$ is $\al$-smashing also) then  the classes 
 $\cu_{w=0}$, $\cu_{[0,1]}$, and (more generally) $\cu_{[m,n]}$ are $\al$-smashing. Moreover,
 the class $\cu_{w\ge 0}$ contains $[ \cup_{i\ge 0}\cu_{w=i}\bral$. 

\item\label{ivneg} 
If $w$ is $\al$-smashing then $\cu_{w=0}$ is an $\al$-negative class. 

\end{enumerate}
\end{pr}
\begin{proof}

Assertions \ref{idual}--\ref{iuni} were proved  in \cite{bws}  (pay attention to Remark \ref{rstws}(3) above!), 
 (the easy) assertions \ref{iwd0}, \ref{ifactp}, and \ref{iwdmod} are given by Proposition 1.2.4(10,8,9) of \cite{bwcp}, and assertion \ref{igenw} easily follows from Propositions 2.4(1) and  3.2(2)  of \cite{bvt}.



Assertion \ref{icoprt} immediately follows from assertion \ref{icopr}.

\ref{ismash}. The first part of the assertion is immediate from the definition of $\al$-smashing weight structures combined with assertion \ref{icopr}.
To obtain the "moreover" part of the assertion one should invoke assertion \ref{iext}.

Assertion \ref{ivneg} is an easy combination of  the Orthogonality axiom in Definition \ref{dwstr}  
  with  assertion \ref{ismash}. 
\end{proof}

We also single out two important (though simple) statements closely related to 
 our proposition.

\begin{pr}\label{prestr}
Assume that $\cu$ is endowed with a weight structure $w$; let $\cu'$ be a full triangulated subcategory of $\cu$.

1. Assume that $\cu'$ is endowed with a weight structure $w'$. Then the embedding $\cu'\to \cu$ is weight-exact with respect to $(w',w)$ if and only if ($w$ restricts to $\cu'$ and) $w'=(\cu_{w\le 0}\cap \obj \cu', \cu_{w\ge 0}\cap \obj \cu')$.

2. $w$ restricts to $\cu'$ whenever $\obj \cu'$ contains either  $\obj\cu_-$ or $\obj \cu_+$. 

\end{pr}
\begin{proof}
1. The "if" implication is immediate from definitions. 

Now we verify the converse implication; so we should check that 
 that $w'=(\cu_{w\le 0}\cap \obj \cu', \cu_{w\ge 0}\cap \obj \cu')$ whenever the embedding $\cu'\to \cu$ is weight-exact. The inclusions $\cu'_{w'\le 0} \subset  \cu_{w\le 0}\cap \obj \cu'$ and $\cu'_{w'\ge 0}\subset \cu_{w\ge 0}\cap \obj \cu'$ follow from the definition of weight-exactness. Thus the converse inclusions follow from Proposition \ref{pbw}(\ref{iort}).

2. Obviously, $w$ restricts to $\cu'$ whenever for any object of $\cu'$ there exists its weight decomposition inside $\cu'$. 

Now, if $\obj \cu'$ contains   
 $\obj\cu_-$  then for any $M\in \obj \cu'$ and any 
$w$-decomposition $w_{\le 0}M\to M\to w_{\ge 1}M$ we have $w_{\le 0}M\in \obj \cu'$; hence this choice of $w_{\ge 1}M$ is $\cu$-isomorphic to an object of $\cu'$ as well.

The proof of the assertion in the case $\obj \cu_{+}\subset \obj \cu'$ is dual to that in the previous case (and the statement itself is dual also; see Proposition \ref{pbw}(\ref{idual})). 
\end{proof}

\begin{rema}\label{rabo} 1. Below we will consider the following subcategory of $\cu$ containing $\cu_-$: in an $\alo$-smashing  $\cu$ 
we will use the notation 
$\cu\abo$ for the $\alo$-localizing subcategory of $\cu$ generated  by $\obj\cu_-$. We will say that the objects of $\cu \abo$ are {\it almost bounded above } (with respect to $w$). 

Proposition \ref{prestr}(2) certainly implies that  $w$ restricts to $\cu\abo$. 

2. We formulate two more easy properties of $\cu\abo$. 

Since $\cu_{w\le m}\perp \obj \cu_{+\infty}$, for any $m\in \z$ we obtain that $ \obj \cu\abo\perp \obj \cu_{+\infty}$ as well.

Moreover,  any left weight-exact functor $F:(\cu,w)\to (\du,v)$ that respects countable coproducts 
 sends almost bounded above objects of $\cu$ into that of $\du$. 

\end{rema}

\section{On 
$\al$-generated and class-generated weight structures}\label{sgenw}

In \S\ref{s2l} we prove two  useful and general (though somewhat technical and easy) statements on triangulated categories that will be used for the construction of weight structures (as well as of so-called weak weight decompositions) 
below.

In \S\ref{salgen} we prove our first main theorem \ref{tvneg} on the existence and properties of  a weight structure $\al$-generated by an $\al$-negative class $\cp\subset \obj \cu$; we also describe some examples and the relation of $\al$-generated weight structures to (general) $\al$-smashing ones.   

In \S\ref{sclgen} we extend earlier results to the setting of class-generated weight structures.  Once again we discuss  examples as well as the relation of class-generated weight structures to earlier constructions of smashing weight structures. 

\subsection{Two useful lemmas}\label{s2l}

To prove the existence of certain weight decompositions we will need the following lemma. 
The formulation of its second part (as well as its proof) closely follows related arguments in several previous papers of the authors. 

\begin{pr}\label{pstar}
Let $P$ and $Q$ be extension-closed $\al$-smashing classes of objects in a triangulated $\al$-smashing category $ \cu$.

1. Then the class $P\star Q$  of all extensions of elements of $Q$ by elements of $P$ is $\al$-smashing as well.

2. Assume that $P\perp Q[1]$. Then $P\star Q$ is extension-closed.

3. Assume  
$P=[\cup_{i\le 0} \cp[i] \bral$ for some $\cp\subset \obj \cu$, and $s\in \z$. Then  $P\perp Q[s]$ whenever $\cp\perp Q[s]$.\footnote{Certainly, one may omit $s$ in this formulation; yet it will be convenient for us to apply it for $s=0$ and $s=1$.}

\end{pr}
\begin{proof}
1. 
Immediate from Proposition \ref{pneem}(1).

2. It suffices to verify the following: for any distinguished triangles $P_1\to E_1\to Q_1\to P_1[1]$, $P_2\to E_2\to Q_2\to P_2[1]$, and $E_1\to E\to E_2\to E_1[1]$ with $P_i\in P$ and $Q_i\in Q$ there exists a distinguished triangle $P_3\to E\to Q_3$ such that $P_3$ is an extension of $P_2$ by $P_1$ and  $Q_3$ is an extension of $Q_2$ by $Q_1$.

Since $P\perp Q[1]$, we have $P_2[-1]\perp Q_1$. 
Hence Proposition 1.1.9 of \cite{bbd} (cf. also Lemma 1.4.1(1) of \cite{bws}) implies 
the existence of  a commutative diagram
$$\begin{CD}
P_2[-1]@>{}>> E_2[-1]@>{}>> Q_2[-1] \\
@VV{}V@VV{}V@ VV{}V \\
P_1@>{}>> E_1@>{}>> Q_1\\
\end{CD}$$
whose rows are the corresponding distinguished triangles. Applying  Proposition 1.1.11 of \cite{bbd} to the left hand square of this diagram we deduce the existence of a commutative diagram 
$$\begin{CD}
P_1@>{}>> E_1@>{}>> Q_1\\
@VV{}V@VV{}V@ VV{}V \\
P_3@>{}>> E@>{}>> Q_3 \\
@VV{}V@VV{}V@ VV{}V \\
P_2@>{}>> E_2@>{}>> Q_2\end{CD}
$$
whose rows and columns are distinguished triangles (for some $P_3,Q_3\in \obj \cu$). Hence we obtain the result.

3. Since the class $\perpp  (Q[s])$ is extension-closed and closed with respect to all coproducts that exist in $\cu$, it contains $P$ whenever it contains $\cp$.

\end{proof}

Now we prove that certain classes of objects are Karoubian  (see \S\ref{snotata}). 

\begin{pr}\label{karcl}
Suppose that $\cu'$ is a full subcategory of $\cu$, $F'$ is an exact endofunctor on $\cu'$,  
 and $\cp$   is an extension-closed subclass of $ \obj \cu'$ such that $F'(\cp)\subset \cp$. 

1. Assume that $\cp[1]\subset \cp$ 
 and  $F' \cong F'[2] \bigoplus \id_{\cu'}$. 
Then $\cp$ is Karoubian. 

2. Dually, $\cp$ is Karoubian whenever  $\cp[-1]\subset \cp$ and   $F' \cong F'[-2] \bigoplus \id_{\cu'}$.

\end{pr}
\begin{proof}
1. The proof is an easy generalization of the arguments used for the proof of \cite[Theorem 3.1]{schnur}.

Certainly we can assume that $\kar(\cu')\cong \cu$ and $\cp$ is 
strict in $\cu$. Then the universal property of the Karoubi envelope construction yields that    $F'$ extends to an exact endofunctor of $\cu$;  we will use the notation $F$ for this extension. We obviously have $F \cong F[2] \bigoplus \id_{\cu}$.

Now let $M\in \cp$, and assume that  $M \cong X \bigoplus Y$ in $\cu$ (for some $X, Y \in  \cu$).
Taking the direct sum of the triangles
$0 \to F(X) \stackrel{\id_{F(X)}}\to F(X) \to 0$,
$$F(X) \to 0 \to F(X)[1] \stackrel{\id_{F(X)[1]}}\to F(X)[1],$$ and 
$F(Y) \stackrel{\id_{F(Y)}}\to F(Y) \to 0 \to F(Y)[1]$,
we obtain the distinguished  triangle 
$$F(M) \to F(M) \to F(X)[1] \bigoplus F(X) \to F(M)[1].$$
It yields that 
$M_1=F(X)[1] \bigoplus F(X)\in \cp$. Similarly, $F(Y)[1] \bigoplus F(Y)\in \cp$; hence $M_2=F(Y)[2] \bigoplus F(Y)[1]$ belongs to $ \cp$ as well.
Moreover, $M_3=X \bigoplus F(X)[2] \bigoplus F(Y) \cong F(X) \bigoplus F(Y) \cong F(M)$ belongs to $\cp$ according to our assumptions.
 
Since $\cp$ is extension-closed, it is also additive; hence the object
$$\begin{gathered}H = M_1\bigoplus M_2\bigoplus M_3\cong X \bigoplus (F(X)[2] \bigoplus F(Y)[2]) \bigoplus (F(Y)[1] \bigoplus F(X)[1]) \\
\bigoplus (F(X) \bigoplus F(Y)) \cong X \bigoplus (F(M)[2] \bigoplus F(M)[1] \bigoplus F(M))\end{gathered}$$ belongs to $\cp$.
Taking the (split)   distinguished  triangle  $$H\stackrel{pr_X}{\to} X\to F(M)[3] \bigoplus F(M)[2] \bigoplus F(M)[1]\to H[1]$$ (where $pr_X$ is the natural projection of $H$ into $X$) we conclude that $X\in \cp$ (since $F(M)[2] \bigoplus F(M)[1] \bigoplus F(M)\in \cp$).

2. Obviously, this assertion is the categorical dual to the previous one.
\end{proof}

\begin{coro}\label{corokar}
1. If $\cp$ satisfies the assumptions of any of the parts of 
 our proposition and is also strict (i.e., closed with respect to  $\cu$-isomorphism) then it is also Karoubi-closed in $\cu$.

2. Assume that $\cu$ 
 is $\al$-smashing for some $\al\ge \alo$. Then for any $\cp\subset \obj \cu$  the classes $C_1=[\cup_{i\ge 0}\cp[i]\bral$ and $C_2=[\cup_{i\le 0}\cp[i]\bral$ are Karoubian and Karoubi-closed in $\cu$.

\end{coro}
\begin{proof}
1. Just note that any strict Karoubian subclass of $\obj \cu$ is obviously Karoubi-closed in $\cu$.

2. Certainly, $C_1[1]\subset C_1$ and $C_2[-1]\subset C_2$. Thus we can take $F=M\mapsto \coprod_{i\ge 0}M[2i]$ (resp. $F=M\mapsto \coprod_{i\le 0}M[2i]$) and apply part 1 (resp. part 2) of Proposition \ref{karcl} to obtain that $C_1$ (resp. $C_2$) is Karoubian. Since these classes are strict in $\cu$ by definition, we also obtain that they are Karoubi-closed in $\cu$. 

\end{proof}

\begin{rema}\label{rkar}
1. The functors $M \mapsto \coprod_{i\ge 0}M[2i]$ and $M \mapsto \coprod_{i\le 0}M[2i]$ appear to be the main examples of the functors that can be taken for $F$ in the context of Proposition \ref{karcl}. Our main reason to consider a "general"  functor $F$ is that these two examples are "not quite self-dual".

2. Recall that in Remark 3.3 of \cite{schnur} it was proved that for an arbitrary additive category $\bu$ the classes $C_1$ and $C_2$ of $\bu$-complexes concentrated in degrees at most $0$ and at least $0$, respectively, are Karoubian (hence, their closures with respect to $K(\bu)$-isomorphisms  are Karoubi-closed in the category $K(\bu)$). 

We (essentially) repeat the proof here. It certainly suffices to prove that $C_1$ is Karoubian since the statement for $C_2$ is its dual. We take $\cu'=K^-(\bu)$ (the homotopy category of bounded above complexes) and $F=M \mapsto \coprod_{i\le 0}M[2i]$.  Then Proposition  \ref{karcl}(1) yields that $C_1$ is Karoubian indeed.

This example explains why we 
 consider two not (necessarily) equal triangulated categories $\cu'\subset \cu$ in Proposition  \ref{karcl} (note that the functor $F$ cannot be extended to $\cu=K(\bu)$ in general).

3. Certainly, we may take $\cp=\obj \cu$ in Corollary \ref{corokar}; thus any $\alo$-smashing triangulated category $\cu$ in Karoubian. So we re-prove Proposition 1.6.8 of \cite{neebook} using a certain Eilenberg swindle argument.

4. If $\cu$ is $\alo$-smashing and a class $\cp\subset \obj \cu$ is $\alo$-smashing and also closed with respect to $[1]$ (resp. $[1]$) then we certainly have $\cp=[\cup_{i\ge 0}\cp[i]\bralo$ (resp. $\cp=[\cup_{i\le 0}\cp[i]\bralo$); thus $\cp$ is Karoubi-closed in $\cu$ according to part 2 of our corollary.
\end{rema}

\subsection{  
Generating weight structures by (countably) negative classes}
\label{salgen}

Till the end of this subsection 
 we will assume that $\cu$ and $\cu'$ are $\al$-smashing triangulated categories for some regular $\al$. 

We prove the basic properties of weight structures that are $\al$-generated by some $
\cp\subset \obj \cu$ (see Definition \ref{dwso} and \S\ref{snotata} for the definitions mentioned in our theorem).

\begin{theo}\label{tvneg}

\begin{enumerate}
\item\label{itvn1}
 Assume that $\al>\alz$ and $\cp$ is an $\al$-negative 
class of objects of an ($\al$-smashing) triangulated category $\cu'$; denote by $\cu$ the $\al$-localizing subcategory of $\cu'$ generated by $\cp$.
 Then there exists an $\al$-smashing weight structure $w$  on $\cu$ that is $\al$-generated by $\cp$. 

\item\label{itvn2}
Conversely, if a  weight structure $w$  on a (triangulated $\al$-smashing) category $\cu$  is $\al$-generated by some $\cp\subset\obj \cu$ then $\cp$ is $\al$-negative.  Moreover, $w$ is left non-degenerate and generated by $\cp$. 
 
\item\label{itvn3}
The heart of the weight structure $w$ in assertion \ref{itvn2} consists of all $\cu$-retracts of coproducts of  families of elements of $\cp$ of cardinality less that $\al$.\footnote{
If $\al \ge \alo$ then nothing will change if we take $\cu'$-retracts here, since $\cu$ is Karoubi-closed in $\cu'$ (apply either  Corollary \ref{corokar}(2) or Remark \ref{rkar}(3)).}

\item\label{itvn4}
 Assume that 
$(\cu,w)$ is 
as in assertion \ref{itvn2}
 and $F:\cu\to \du$ is an exact functor respecting coproducts of less than $\al$ objects, where $\du$ is a triangulated category 
 endowed with a weight structure $v$. Then $F$ is left weight-exact if and only if $F(\cp)\subset \du_{v\le 0}$.

Moreover, if $v$ (along with $\du$) is $\al$-smashing then $F$ is right weight-exact if and only if $F(\cp)\subset \du_{v\ge 0}$.

\item\label{itvn5}
 Under 
 the assumptions of the previous assertion suppose in addition that $F$ is weight-exact 
and surjective on objects, and $\al>\alz$. 
Then the weight structure $v$ is $\al$-generated by $F(\cp)$. 

In particular, $w$ is the only $\al$-smashing weight structure on $\cu$ whose heart contains $\cp$.


\item\label{itvn6}
  Assume that $\al\ge \alo$ and $\cu$ is an $\al$-smashing triangulated category.
Then a class $\cp\subset  \obj \cu$ is $\al$-negative 
if and only if the class $\cp\tal$ consisting of all coproducts of families of elements of $\cu$ of cardinality less than $\al$ is countably negative. 
\end{enumerate}

\end{theo}
\begin{proof}

\ref{itvn1}. We should prove that the couple $([\cup_{i\ge 0}\cp[i]\bral,[\cup_{i\le 0}\cp[i]\bral)$ satisfies the axioms of a weight structure on $\cu$ (recall that $[\cup_{i\ge 0}\cp[i]\bral$ is $\al$-smashing). These classes are Karoubi-closed in $\cu$ according to Corollary \ref{corokar}(2). The corresponding semi-invariance with respect to shifts properties are obvious. Since $\cp$ is $\al$-negative, the orthogonality axiom 
 follows from Proposition \ref{pstar}(3) (we take $P=[\cup_{i\le 0}\cp[i]\bral$, $Q=[\cup_{i>0}\cp[i]\bral$, and $s=0$ in it). Parts 1 and 2 of the proposition imply that the class $ [\cup_{i\le 0}\cp[i]\bral \star ([\cup_{i\ge 0}\cp[i]\bral [1])$ (where the operation $-\star-$ is defined in Proposition \ref{pstar}(1)) is extension-closed and $\al$-smashing; hence it coincides with $\obj \cu$ and we obtain the weight decomposition axiom.


\ref{itvn2}. Once again,  if $w$ is $\al$-generated by  $\cp$ then $w$ is  $\al$-smashing. Hence $\cu_{w=0}$ is $\al$-negative (see Proposition \ref{pbw}(\ref{ivneg})); thus $\cp$ also is. 

Now let $M$ be a left weight-degenerate object of $\cu$; consider the class $P=\perpp M$. Since $\cp[i]\subset \cu_{w=i}$ for all $i\in \z$, the orthogonality axiom for 
$w$ implies that $P$ contains $\cup_{i\in \z}\cp[i]$. Since $P$ is obviously $\al$-smashing and extension-closed, it also contains $M$; 
  hence $M=0$.\footnote{This simple argument probably originates from \cite{neebook}.} 

Lastly, the class $R= (\cup_{i<0}\cp[i])\perpp$  certainly contains $\cu_{w\ge 0}$. So for $N\in R$ we should verify that $N\in \cu_{w\ge 0}$. 
 Once again we note that the class $\perpp N$ is $\al$-smashing and extension-closed; since it contains $\cp[i]$ for all $i<0$, it also contains  $[\cup_{i<0}\cp[i]\bral=\cu_{w\le 0}$. Hence $N$ belongs to $\cu_{w\ge 0}$ (see  Proposition \ref{pbw}(\ref{iort})). 
Thus $R=\cu_{w\ge 0}$ indeed.

\ref{itvn3}. Denote  our candidate for $\cu_{w=0}$ by $C$.
Firstly we note that  $C\subset \cu_{w=0}$ since the latter class is $\al$-smashing (see Proposition \ref{pbw}(\ref{ismash})) and Karoubi-closed in $\cu$.

Applying Proposition \ref{pbw}(\ref{isplit}) 
 we obtain that $C$  is  extension-closed (since this class is certainly additive); certainly, it is also $\al$-smashing.  

Next, the class $C \star \cu_{w \ge 1}$ is   extension-closed and $\al$-smashing according to  Proposition \ref{pstar}(1,2). Hence this class coincides with $[\cp\cup (\cup_{i> 0}\cp[i])\bral= \cu_{w \ge 0}$. Thus for any $M\in \cu_{w= 0}$ there exists its weight decomposition $LM\to M\to RM\to LM[1]$ with $LM\in C$.  Since $M$ is a retract of $LM$ (see Proposition \ref{pbw}(\ref{iwdmod})) we obtain that $M\in C$.

\ref{itvn4}. Since $F$ respects coproducts of less than $\al$ objects, we have $F(\cu_{w \le 0})\subset [\cup_{i\le  0}F(\cp)[i]\bral$  and $F(\cu_{w \ge 0})\subset [\cup_{i\ge  0}F(\cp)[i]\bral$. It remains to recall that $\du_{v\le 0}$ is closed with respect to $\du$-coproducts  (see Proposition \ref{pbw}(\ref{icopr})), and  $\du_{v\le 0}$ is $\al$-smashing whenever $v$ is  (by definition).

\ref{itvn5}. 
According to Proposition \ref{pdesc}(1) below we have $\du_{v\le 0}=\kar_{\du}(F( [\cup_{i\le  0}\cp[i]\bral))$ and  $\du_{v\ge 0}=\kar_{\du}(F( [\cup_{i\ge  0}\cp[i]\bral))$.
Since $F$ respects coproducts and $w$ is $\al$-smashing, we obtain that $v$ is $\al$-smashing as well. Hence $D_1=
[\cup_{i\le  0}F(\cp)[i]\bral \subset \du_{v\le 0}$ and  $D_2=
[\cup_{i\ge  0}F(\cp)[i]\bral\subset \du_{v\ge 0}$  according to Proposition \ref{pbw}(\ref{icoprt},\ref{ismash}). 

Now, the classes $D_i$ ($i=1,2$) are  Karoubi-closed in $\du$ (see Corollary \ref{corokar}(2)); thus they coincide with   $\du_{v\le 0}$ and  $\du_{v\ge 0}$, indeed.

We conclude by noting that the "in particular" part of our assertion immediately follows from the "main" part if we take $F$ to be the identity of $\cu$. 

\ref{itvn6}. The "only if" implication is obvious; the converse one is  given by Corollary \ref{cabo}(I.2) below.  

\end{proof}

Now we discuss examples to our theorem and its actuality.

\begin{rema}\label{rcompar}
\begin{enumerate}

\item\label{imax} 
It is certainly interesting to know which $\al$-smashing weight structures are $\al$-generated 
  by some   classes of objects. Obviously, if $w$ is $\al$-generated by some $\cp$ then it is also $\al$-generated by $\cu_{w=0}$. So, the question is whether $\cu=\lan \cu_{w=0}\ral$. 
	We don't know any "complete" answer to this question; so here we will only note that this is 
the case if  $\cu=\lan P\ral$ 
 for some class $P\subset \obj \cu_b$ (according to Proposition \ref{pbw}(\ref{iextcub})). 

If we assume in addition that $w$ is generated by this $P$ then 
 this condition may be re-formulated as follows: for each $M\in P$ there exists $i_M\in \z$ such that $P\perp M[i]$ for all $i>i_M$ (cf. 
 Corollary \ref{classgws}(3) below). 

Some more results in this direction may be found in Corollary \ref{cabo}(II.2) and Remark \ref{rconj}(1) below.

\item\label{imorexamples} 
Now we describe rather explicit examples of $\cp\subset \obj\cu'$ such that $\cp$ satisfies the $\al$-negativity 
 condition. 

Certainly, for any $\al$-smashing additive category $\bu$ the class $\obj \bu$ is $\al$-smashing in the category $\cu'=K(\bu)$, and part \ref{itvn4} of our theorem implies that the corresponding $w$ is the restriction to $\cu$ of the 
weight structure $\wstu$ on $\cu'$ (see Remark \ref{rstws}(1)). 

One can easily generalize this example by choosing some $d\ge 0$ and 
 a {\it differential graded} category $C$ such that $C^i(M,N)=\ns$ for any $M,N\in \obj \bu$ whenever $i>0$ or $i<-d$ (see \S6.1 of \cite{bws} for the notation) that is $\al$-smashing in the obvious sense. 
 We take $\cu'$ to be the homotopy category of the easily defined category of {\it unbounded twisted complexes}\footnote{It appears that twisted complexes were originally defined in \cite[\S8]{tlt}, and they were related to differential graded categories in \cite{bk}.} over $C$ (cf. loc. cit.;  we assume $C^i(M,N)=\ns$ for $i<-d$  to ensure that unbounded twisted complexes can be defined without any problem). Then the 
  homotopy category of $C$ is a full subcategory of $\cu'$ that easily seen to be $\al$-negative in it.

Lastly, one can obtain plenty of examples of $\al$-generated weight structures via weight-exact localizations; see 
Theorem \ref{tloc} and Remarks \ref{rprloc} and \ref{rrheart} below.


\item\label{ifiltr}
For any $\al$-smashing weight structure $w$ on $\cu$, any regular infinite $\be\le \al$, and any $\cp\subset \cu_{w=0}$ part \ref{itvn1} of our theorem gives a weight structure 
 on the category $\lan \cp \ra^{\be}$. Moreover, Proposition \ref{prestr}(1)  (along with part \ref{itvn4} of our theorem) implies that this weight structure is the restriction of $w$ to   $\lan \cp \ra^{\be}$. So, if we fix $\cp$ and vary $\be$ then we obtain a filtration of $\cu$ by subcategories such that $w$ restricts to them. 
As we will demonstrate 
in Remark \ref{rcgws}(\ref{itop}) below, 
it does make sense to take $\cp\neq  \cu_{w=0}$ here.

\item\label{izneg}
 Certainly, a class $\cp\subset \obj \cu$ is $\alz$-negative if and only if $\cp\perp\cp[i]$ for all $i\ge 0$ (we will also say that $\cp$ is just {\it negative} if the latter condition is fulfilled). The authors doubt that any similar easy criterion 
is available for 
$\al>\alz$ (yet see  Theorem \ref{tvneg}(\ref{itvn6})). 
However, we will mention a certain condition that implies $\al$-negativity in Remark \ref{rpperf} below.

\item\label{ifiltrb} The filtration on an  
 $\al$-smashing category $\cu$ generated by a given 
$\al$-negative $\cp\subset \obj \cu$ (see part \ref{ifiltr} 
of this remark) can easily be "completed" by terms corresponding to those cardinals that are not regular (and less than $\al$). Indeed, to any cardinal $\be$ (that is less than $\al$) one can associate the subcategory $\cu_{\cp}^{\be}=\cup_{\gamma\le \be}\lan \cp \ra^{\gamma}$ 
 (here $\gamma$ runs through regular cardinals that are not greater than $\be$). 
 $\cu_{\cp}^{\be}$ certainly coincides with $\lan \cp \ra^{\be}$ if $\be$ is regular, and the weight structure $w$ 
  $\al$-generated by $\cp$ obviously restricts to $\cu_{\cp}^{\be}$.

So, the only "problem" with $\cu_{\cp}^{\be}$ is that it does not have to be $\be$-smashing (unless $\be$ is regular). 

\item\label{ialz} 
Lastly, we wouldn't have to exclude the case $\al=\alz$ in parts \ref{itvn1} and \ref{itvn5} of our theorem if we have had   modified our definition of the weight structure $\alz$-generated by $\cp$ as follows: $w=(\kar_{\cu}([\cup_{i\le 0}\cp[i]]_{\alz}), \kar_{\cu}([\cup_{i\ge 0}\cp[i]]_{\alz})) $; see Corollary 2.1.2 of \cite{bonspkar}. 
	 \end{enumerate}
	
\end{rema}

\subsection{ On class-generated weight structures}\label{sclgen}

Due to the importance of the smashing case of our statements we treat them separately (though this is not absolutely necessary; cf. Remark \ref{rclass}(2)).
So till the end of this section we will assume that $\cu$ is smashing.

The following corollary is possibly even more actual than Theorem \ref{tvneg} itself. Once again, one may consult Definition \ref{dwso} and \S\ref{snotata} for the definitions mentioned in it.

\begin{coro}\label{classgws}
1. The natural "smashing" versions of all of the assertions of Theorem \ref{tvneg} are fulfilled.

In particular, if $\cu$ is generated by its class-negative class of objects $\cp$ as its own localizing subcategory then there exists a smashing weight structure $w$ on $\cu$ that is class-generated by $\cp$. Moreover,   $w$ is generated by $\cp$,  
 $w$ is the only smashing weight structure on $\cu$ whose heart contains $\cp$, and $\cu_{w=0}$ 
consists of all $\cu$-retracts of coproducts of  families of elements of $\cp$. 
Furthermore,  
for a smashing weight structure $v$ on (a smashing triangulated category) $\du$ an exact  functor $F:\cu\to \du$ that respects coproducts is left (resp. right) weight-exact 
 if and only if it sends $\cp$ into $\cu_{w\le 0}$ (resp. into $\cu_{w\ge 0}$).

2. For any infinite regular $\al$ 
 this weight structure $w$ restricts to an $\al$-smashing weight structure $w^{\al}$ that is $\al$-generated by $\cp$ on the category $\lan \cp \ral$.
Furthermore, for any (not necessary regular) cardinal $\be$ our $w$  restricts to the triangulated subcategory of $\cu$ whose object class equals 
$\cu_{\cp}^{\be}=  \cup_{\gamma\le \be}\obj(\lan \cp \ra^{\gamma})$ 
 (here $\gamma$ runs through regular cardinals that are not greater than $\be$). 

3. A smashing weight structure $w$ on $\cu$ is class-generated by $\cu_{w=0}$ whenever  $\cu=\lan \cq\rab$  for some $\cq\subset \obj \cu_b$. 
Moreover, if  $w$ is generated by $\cq$ then it suffices to assume that for each $M\in \cq$ there exists $i_M\in \z$ such that $\cq\perp M[i]$ for all $i>i_M$. 

\end{coro}
\begin{proof}
1. It is easily seen that this statement follows from Theorem \ref{tvneg} applied for $\al$ running through all infinite regular cardinals. 
Indeed, a weight structure (and a triangulated category) is smashing if and only if  it is $\al$-smashing for all $\al$, $\cp$ is class-negative if and only if it is $\al$-negative for all $\al$, etc., and it remains to
note that   for any $I\subset \z$ (the "actual" cases here are $I=\z$, $I=\n$, and $I=-\n$)  the class $[(\cup_{i\in I}\cp[i])\brab$ is the union of $[(\cup_{i\in I}\cp[i])\bral$ for all (infinite regular)  $\al$ (since this union is obviously closed with respect to $\cu$-coproducts).  

Alternatively,  one can  modify the proof Theorem \ref{tvneg}  in the obvious way, or pass to a larger universe (cf. Remark \ref{rclass}(2)).

2. Both statements follow from our definitions 
  along with Theorem \ref{tvneg}(\ref{itvn1}) immediately.

3. To prove the first statement it suffices to recall that $\cq\subset \cu_b=\lan \cu_{w=0}\ra$  (see Proposition \ref{pbw}(\ref{iextcub})).
Next, if $\cq$  generates $w$ then  $\cq\subset \cu_{w\le 0}$ (see Proposition \ref{pbw}(\ref{igenw})), whereas the assumption on the existence of $i_M$ is equivalent to  $M\in \cu_{w\ge -i_M}$ according to Proposition \ref{pbw}(\ref{iort}). 

\end{proof}

Now we discuss the assumptions of our corollary and examples for it.

\begin{rema}\label{rcgws}
\begin{enumerate}
\item\label{ibger} The assumptions on $\cq$ in part 3 of the corollary 
 are fulfilled for the categorical duals of all the  Gersten weight structures studied in \cite{bgn} (whereas these Gersten weight structures themselves are {\it cosmashing} and {\it cocompactly cogenerated}); see Theorems 3.3.2(7) and 5.1.2(II.6) of ibid. 

Now we describe 
 a simple example for our corollary that is actual for \cite{bokum} and \cite{bwcp}. For a smashing additive category $\bu$ one can take the smashing category $\cu'=K(\bu)$, consider its localizing subcategory $\cu$ generated by $\obj \bu$, and the weight structure $w$ class-generated by $\obj \bu$ in $\cu$. 
 $w$ is the only smashing weight structure on $\cu$ whose heart contains $\obj \bu$ (see part 1 of our corollary). On the other hand, the stupid weight structure on $\cu'$ (see Remark \ref{rstws}) is  smashing on   $\cu'$; it restricts to $\cu$ since $\obj\cu$ contains $\cu'_{\wstu\ge 0}$ (see Proposition \ref{prestr} and Remark \ref{rconj}(2)). Thus $w$ equals the restriction of $\wstu$ to $\cu$. 

Another interesting observation is that this category $\cu'$ is easily seen not to contain non-zero  right or left weight-degenerate objects. Moreover, $\cu'$ is "usually" distinct from $\cu$. For instance, if $\bu$ is the category of free $\z/4\z$-modules then the complex $M=\dots \stackrel{\times 2}{\to}\z/4\z\stackrel{\times 2}{\to}\z/4\z\stackrel{\times 2}{\to}\z/4\z\stackrel{\times 2}{\to}\dots$ is non-zero and it is not contained in the corresponding $\cu\subset \cu'=K(\bu)$ since $\z/4\z[i]\perp M$ for all $i\in \z$ (and it follows that $\obj \cu\perp_{\cu'} M$). Note also that $\cu'$ is {\it $\alo$-well-generated} in this case, see Theorem 5.9 of \cite{neeflat}.

Furthermore, this example 
 can be easily generalized using twisted complexes over negative differential graded categories (cf. Remark \ref{rcompar}(\ref{imorexamples})). 

\item\label{icomp} 
Recall that an object $M\in \obj \cu$ is said to be {\it compact} if the functor  $\cu(M,-):\cu\to \ab$ respects coproducts.
Certainly, if $\cp$ consists of compact objects then it is class-negative if and only if it is negative (see Remark \ref{rcompar}(\ref{izneg})). Negative classes of compact objects yield an important family of examples for our theorem (note here that the clumsy  argument used for the proof of the corresponding particular case of our theorem in \cite[Theorem 4.5.2]{bws} can be used only if we have a negative {\bf set} of compact generators, and passing to a larger universe does not help here).

 Moreover, one can obtain plenty of examples as follows:  
an arbitrary set of compact objects in $\cu$ 
generates a smashing weight structure  $w$ (see Theorem 5 of \cite{paucomp} or Remark 4.2.2(1) of \cite{bpure});
weight structures obtained this way are said to be {\it compactly generated}.  
  Next, $w$ restricts to the weight structure class-generated by $\cu_{w=0}$ on the localizing subcategory of $\cu$ that is class-generated by $\cu_{w=0}$ (see 
  Corollary \ref{classgws}(1)).

\item\label{itop} Now we describe an important example  to our corollary in detail. 

We take $\cu=\shtop$ (the stable homotopy category) and $\cp=\{S\}$ (the sphere spectrum). Then $S$ is compact, $\cp$ is negative in  $\shtop$ 
 and also generates $\shtop$ as its own localizing subcategory. Hence $\cp$ is also class-negative; thus  there exists a weight structure class-generated by $\cp$ on $\shtop$. 
We also obtain that $w$ is generated by $\cp$.
Hence $w$ coincides with the {\it spherical} weight structure $w\sph$ on $\shtop$; this weight structure was studied in \S4.6 of \cite{bws},   \S4.2 of \cite{bwcp}, and \S4.2 of \cite{bkwn}. 
Thus we obtain $\shtop_{w\sph\le 0}=[\cup_{i\le 0}\cp[i]\brab$. 

Moreover,  we obtain a filtration of $\shtop$ by $\lan \cp\ral$ (and by $\shtop_{\cp}^{\be}$) such that  $w\sph$ restricts to its levels. 
Proposition 3.2.10 of \cite{marg} easily implies that for  any $\be > \alz$ a spectrum $P$ belongs to $\shtop_{\cp}^{\be}$ (recall that this 
 class equals $\obj(\lan \cp \ra^{\be})$ if $\be$ is regular)  if and only if all the homotopy groups of $P$ are of cardinality less than $\be$; one should only take into account 
 that  {\it crude Postnikov towers} mentioned in loc. cit. are defined in terms of 
 the corresponding {\it weak homotopy colimits}, and apply Proposition 3.1.4 of ibid. (cf. Definition \ref{dcoulim} below). 
  Furthermore, the category  $\lan \cp\ra^{\alz}$ 
consists of   finite spectra; cf. \S4.6 of \cite{bws}.

	Lastly we note that part 2 of our corollary gives the following nice weight-exactness criterion for $w\sph$:  an exact functor $F:\cu\to\du$ as in the corollary (i.e., $F$ respects coproducts and $\du$ is endowed with a smashing weight structure $v$) is left (resp., right) weight-exact  if and only if $F(S)\in \du_{v\le 0}$ (resp. $F(S)\in \du_{v\ge 0}$). 
	
Furthermore, all these statements can be naturally extended to the setting equivariant stable homotopy categories, i.e., we take $\cu=SH(G)$, where $G$ is a compact Lie group; see	\S4.1 of \cite{bwcp} for more detail.

\item\label{iperf}
Recall that a class $\cp\subset \obj \cu$ is said to be  {\it perfect} (cf.  Remark \ref{rpperf} below) if 
the morphism class $$\operatorname{Null}(\cp)= \{f\in \mo(\cu);\ \cu(P,-)(f)=0\ \forall P\in \cp\}$$ is closed with respect to coproducts (cf. Remark 4.3.5 of \cite{bpure} where this definition of perfect classes is shown to be closely related to other versions of this notion, which essentially originates from \cite{neebook}); $\cp$ is said to be {\it countably perfect} if the class 
$\operatorname{Null}(\cp)$ is closed with respect to countable coproducts. Certainly, any class of compact objects (of $\cu$) is perfect.

Now, Theorem 4.3.1 of \cite{bpure} says that any countably perfect {\bf set} of objects generates some weight structure; this result appears to be the most general statement on the construction of unbounded weight structures.\footnote{Recall that all bounded weight structures are $\alz$-generated by negative classes of objects; see Proposition \ref{pbw}(\ref{iextcub}) and Remark \ref{rcompar}(\ref{izneg}).} So it is an actual question which weight structures constructed by means of the results of the current paper cannot be obtained from that theorem.

Firstly we note that  loc. cit. cannot be applied if $\cu$ is not smashing (in contrast to Theorem \ref{tvneg}); moreover, it appears to be quite difficult to extend the arguments used in the proof of loc. cit. to the non-smashing setting.\footnote{ In particular, loc. cit. does not yield the aforementioned restrictions of  $w\sph$ to $\lan \{S\}\ral$; yet (in  this particular case) the existence of these restrictions appears to be rather easy anyway.} 

Another way to obtain examples that are not perfectly generated is to consider "big ascending systems" 
 of triangulated categories. In particular, if $\{\cu_i:\ i\in I\}$ is a proper class of non-zero smashing  triangulated categories that are generated by 
 class-negative 
 classes $\cp_i\subset \obj \cu_i$  then one can take the category $\cu$ whose objects are those families $(M_i:\ M_i\in \obj \cu_i)$ 
such that $\{i\in I;\ M_i\neq 0\}$ is (only) a set. This $\cu$ is easily seen to be 
generated by the corresponding class-negative 
 class $\cp$ as its own localizing subcategory; yet it is easily seen not to be generated by any {\bf set} of objects (as its own localizing subcategory). 
Thus the corresponding weight structure $w$ is not generated by any set of objects either. 
 One may say that 
 Corollary \ref {classgws} can be used for treating weight structures on "really big" triangulated categories 
 in contrast to Theorem 4.3.1 of \cite{bpure}. 
Note also that no version of the pseudo-perfectness  assumption on $\cu$ (see Remark \ref{rpperf} below) 
 is sufficient to make the proof of loc. cit. work.

\item\label{ir4} 
Our calculation of $\cu_{w=0}$ for a class-generated 
 weight structure $w$  
yields that $w$ coincides with the weight structure class-generated by some $\cp\subset \cu_{w=0}$ (certainly, this is the case if and only if $\cu=\lan \cp \rab$) if and only if any element of  $\cu_{w=0}$ is a retract of a coproduct of a family of elements of $\cp$. 
Hence it is rather reasonable to ask when one can choose a {\bf set} $\cp$ satisfying this condition. 
 So we recall the following statement: a set $\cp$ satisfying our conditions exists whenever $\cu$ is generated by some set of its objects as its own localizing subcategory (see 
 Proposition 2.3.2(9) of \cite{bwcp}).\footnote{Moreover, the obvious $\al$-generated case of this statement is also valid; see 
 Proposition 2.3.4(3) of ibid.} 
\end{enumerate}
\end{rema}

Let us also mention an obvious statement on (co)homological functors that respect coproducts; it will be applied in other papers.

We recall that an additive covariant (resp. contravariant) functor from $\cu$ into $\au$ is said to be {\it homological} (resp. {\it cohomological}) if it converts distinguished triangles into long exact sequences.

\begin{lem}\label{lcc}
Assume that $\cu$ is endowed with a weight structure $w$ that is class-generated by a class $\cp$.

Let $H$ be a homological (resp., cohomological) functor from $\cu$ into an abelian category $\au$ that respects coproducts (resp. converts them into products).

Then $H$ annihilates $\cu_{w\le 0}$ (resp. $\cu_{w\ge 0}$) if an only if it kills $\cp[i]$ for all $i\le 0$ (resp. $i\ge 0$).  
\end{lem}
\begin{proof}
The statement is very easy. 

Recall that the class $\cu_{w\le 0}$ (resp. $\cu_{w\ge 0}$) contains  $\cup_{i\le 0}\cp[i]$ (resp. $\cup_{i\ge 0}\cp[i]$; see Definition \ref{dwso}(\ref{idcgen})); thus we obtain the "only if" statement. 

Next, our assumptions on $H$ yield that the class $\{M\in \obj \cu:\ H(M)=0\}$ is smashing and extension-closed. Thus recalling Definition \ref{dwso}(\ref{idcgen}) once again we obtain the "if" implication.
 
\end{proof}

\begin{rema}\label{rccal}
Obviously, the $\al$-generated version of this statement is valid as well. 
\end{rema}

\section{On weight-exact localizations}\label{sloc}

This section is dedicated to those localization functors $\pi:\cu\to \cu'$ that are weight-exact (with respect to some $(w,w')$; we will say that $w$ {\it descends} to $\cu'$ if such a weight structure $w'$ exists). We start with some statements concerning weight-exact functors that are essentially surjective on objects. Next we prove Theorem \ref{tloc} that includes both an if and only if criterion for the weight-exactness of $\pi$ (and so, for the existence  of the corresponding $w'$) and a construction of a vast family of weight-exact localizations (that contains the families considered in \cite{bososn}). 

In \S\ref{ssheart}  we calculate the heart of the weight structure $w'$ constructed by means of the aforementioned theorem. These statements are natural generalizations of the corresponding results of  \cite{bososn} (and the remaining "calculational" statements automatically extend to our setting also); in particular, $\hw'$ can be obtained from $\hw$ by means of a certain universal construction that mentions additive categories only. 

In \S\ref{sgenloc} we discuss the relation of weight-exact localizations to {\it generalized universal localizations} of rings. 

\subsection{
Criteria for  weight-exactness of  localizations}\label{scrit}

We start from the following general statement.

\begin{pr}\label{pdesc}
Let $\cu$ and $\cu'$ be triangulated categories endowed with weight structures  $w$ and $w'$, respectively; let $F:\cu\to \cu'$ be an exact functor that is essentially surjective on objects.

1. Then $F$  is weight-exact with respect to $(w,w')$ if and only if $w'=(\kar_{\cu'}(F(\cu_{w\le 0})), \kar_{\cu'}(F(\cu_{w\ge 0})))$.

2. Moreover, if this is the case then we also have $\cu'_{[m,n]}=\kar_{\cu'}(F(\cu_{[m,n]}))$ for any $m\le n \in \z$.

3. Furthermore, if $F$ is weight-exact and respects coproducts of less than $\al$ objects (resp., all small coproducts)  and $w$ is $\al$-smashing (resp. smashing) then $w'$ also is.

4.  Assume that $F$ is weight-exact; let $G:\cu'\to \du$ be an  
exact functor, where $\du$ is a triangulated category endowed with a weight structure $v$, and assume that $G\circ F$ is weight-exact (with respect to $(w,v)$). Then $G$ is weight-exact (with respect to $(w',v)$) as well. 

\end{pr}
\begin{proof}
1. If  $w'=(\kar_{\cu'}(F(\cu_{w\le 0})), \kar_{\cu'}(F(\cu_{w\ge 0})))$ then $F$ is weight-exact by definition. 

 Conversely, if $F$ is weight-exact with respect to $(w,w')$ then we certainly have $\kar_{\cu'}(F(\cu_{w\le 0}))\subset \cu'_{w'\le 0}$ and  $\kar_{\cu'}(F(\cu_{w\ge 0}))\subset \cu'_{w'\ge 0}$. According to Proposition \ref{pbw}(\ref{iuni}) it remains to verify that 
 $(\kar_{\cu'}(F(\cu_{w\le 0})), \kar_{\cu'}(F(\cu_{w\ge 0})))$ is a weight structure on $\cu'$ as well. Now, the orthogonality axiom for this couple follows from that for $w'$, whereas the remaining axioms are obvious (in particular, one can take $\cu'$-weight decompositions "coming from $\cu$"; cf. the proof of assertion 2).

2. Certainly, $\kar_{\cu'}(F(\cu_{[m,n]}))\subset \cu'_{[m,n]}$.
The proof of the converse inclusion is a simple application of the axioms of weight structures. 

For $N\in \cu'_{[m,n]}$ and $M\in \obj \cu$ such that $F(M)\cong N$ we apply $F$ to an $n$-weight decomposition of $M$ to obtain a $w'$-$n$-decomposition distinguished triangle
$$F(w_{\ge n+1}M)[-1]\to F(w_{\le n}M)\to N\to F(w_{\ge n+1}M).$$ Since $N\perp F(w_{\ge n+1}M)$, we obtain that 
 $F(w_{\le n}M)\cong N\bigoplus F(w_{\ge n+1}M)[-1]$.
Since $F(w_{\ge n+1}M)[-1]$ is a retract of $F(w_{\le n}M)$, we have $F(w_{\ge n+1}M)\in \cu'_{w'=n+1}$. Thus  $F(w_{\le n}M)\in \cu_{[m,n]}$ and it suffices to verify 
that $N'=F(w_{\le n}M)$ is a retract of an element of $F(\cu_{[m,n]})$.

Now we apply an "almost dual" argument to $N'$. For $M'=w_{\le n}M$ we take an $m-1$-weight decomposition of $M$ and apply $F$ to it to obtain the following triangle: $F(w_{\le m-1}M')\to N'\to F(w_{\ge m}M')\to F(w_{\le m-1}M')[1]$. Since $F(w_{\le m-1}M')\perp N'$, $N'$ is a retract of $F(w_{\ge m}M')$. It remains to recall that $w_{\ge m}M'\in \cu_{[m,n]}$ according to  Proposition \ref{pbw}(\ref{iwd0}).

3. We should prove that $\cu'_{w'\ge 0}$ is closed with respect to coproducts of less than $\al$ objects (resp., all coproducts). Since $\cu'_{w'\ge 0}=\kar_{\cu'}(F(\cu_{w\ge 0}))$, this statement follows from the corresponding property of $w$ (note that a coproduct of retracts of $F(M_i)$ for $M_i\in \cu_{w\ge 0}$ is a retract of $\coprod F(M_i)$). 

4. Obvious from the description of $w'$ given by assertion 1.
\end{proof}

\begin{rema}\label{rdesc}
Part 1 of our proposition  certainly means that there is at most one weight structure $w'$ such that $F$ is weight-exact; this weight structure exists whenever  $(\kar_{\cu'}(F(\cu_{w\le 0})), \kar_{\cu'}(F(\cu_{w\ge 0})))$ is a weight structure on $\cu'$.

In the latter case we will say that $w$ {\it descends} to $\cu'$, and the corresponding $w'$ is the {\it descended} weight structure. 

Below we will be interested in the case where $F$ is some Verdier localization functor (certainly, $F$ is surjective on objects in this case).

We will usually use the following notation: $\du$ is a full triangulated subcategory of $\cu$, $\cu'=\cu/\du$, $\pi:\cu\to \cu'$ is the localization functor,  $w'$ denotes the couple $(\kar_{\cu'}(\pi(\cu_{w\le 0})), \kar_{\cu'}(\pi(\cu_{w\ge 0})))$ whether this is a weight structure on $\cu'$ or not.

Some more comments to our proposition are given in Remark \ref{rrdesc} below.

\end{rema}
 
Now we prove our first theorem about weight-exact localizations.

\begin{theo}\label{tloc}
Assume that $w$ is a weight structure on $\cu$, and  $\du$ is its full triangulated subcategory such that the Verdier localization $\cu'=\cu/\du$ exists (i.e., all its Hom-classes are sets; see Remark \ref{rclass}(3)). 

1. 
  Then $w$ descends to 
	 a weight structure $w'$ on  $\cu'$ (see Remark \ref{rdesc}) if and only if  for any $N\in \obj \du$ there exists a $\du$-distinguished triangle
 \begin{equation}\label{ewwd}
D(N)=(LN \stackrel{i}{\to} N\stackrel{j}{\to} RN\to LN[1])
\end{equation}
such that $i$ factors through (an element of) $\cu_{w\le 0}$ and $j$ factors through $\cu_{w\ge 0}$. 

2. The assumption of assertion 1 is fulfilled whenever  for any $N\in \obj \du$ there exists a distinguished triangle 
\begin{equation}\label{ewwdu}
L_{\du}N \stackrel{i}{\to} N\stackrel{j}{\to} R_{\du}N\to L_{\du}N[1]\end{equation}
with 
$L_{\du}N\in \obj \du\cap \cu_{w\le 0}$ and $R_{\du} N\in \obj \du\cap \cu_{w\ge 0}$. 

If this condition is fulfilled then we will say that there exist {\it weak 
$w$-decompositions inside $\du$}.

3. 
Weak $w$-decompositions exist inside $\du$  if either

 (i) there exists a regular (infinite) cardinal $\al$ such that $w$ is $\al$-smashing (and so, $\cu$ also is) and $\du=\lan B \ral$, where 
 $B=B_0\cup B_1 \cup B_2$, and $B_0$ consists of cones (in $\cu$) of a class $S$ of $\hw$-morphisms, elements of $B_1$ are left $w$-degenerate, and elements of $B_2$ are right $w$-degenerate. 

(ii) $\cu$ and $w$ are smashing and $\du=\lan B \rab$  for $B$ as above.\footnote{Once again, the definition and notation used in assertion 3 can be found in \S\ref{snotata} and Definition \ref{dwso}.} 

Moreover, in case (i) the functor $\pi$ respects coproducts of less than $\al$ objects, and in case (ii) $\pi$ respects all coproducts (thus the category $\cu'$ in these cases is closed with respect to coproducts of the corresponding sorts), and the descended weight structure $w'$ is $\al$-smashing (resp., smashing).
\end{theo}
\begin{proof}
Recall that we write $\pi$ for the localization functor $\cu\to \cu'$.

1. Assume that 
$D(N)$ exists for any $N\in \obj \du$. Then we check that the couple $w'$ described in Remark \ref{rdesc} is a weight structure on $\cu'$.

Obviously, $\cu'_{w'\ge 0}[1]\subset \cu'_{w'\ge 0}$, and $\cu'_{w'\le 0}\subset \cu'_{w'\le 0}[1]$;  both $\cu'_{w'\le 0}$ and $\cu'_{w'\ge 0}$ are Karoubi-closed in $\cu'$ by construction. Moreover, any object of $\cu'$ possesses a "decomposition with respect to $w'$" obtained by applying $\pi$ to a weight decomposition of its preimage in $\cu$.

Hence it remains to prove the orthogonality axiom. It certainly suffices to verify that  $\pi(\cu_{w\le 0})\perp_{\cu'} \pi(\cu_{w\ge 1})$.

So, let $\phi \in \cu' (X,Y)$, where $X \in \cu_{w\le 0},\ Y \in \cu_{w\ge 1}$ (here we use the same notation for objects of $\cu$ and for their images in $\cu'$). By the theory of Verdier localizations, $\phi$ can be presented as a composition $f\circ s^{-1}$, where $f\in \cu(T,Y), s\in \cu(T,X)$ for some object $T$ of $\cu$ such that $E=\co(s)[-1] \in \obj\du$.
Denote by $p$ the cone map $X \to E[1]$ and by $q$ the composed morphism $E \to Y$. Consider a 
decomposition of the type (\ref{ewwd}) for $E$: 
$$D(E)=(LE \stackrel{i}\to E \stackrel{j}\to RE \to LE[1]).$$

The composition $j[1]\circ p$ is zero since it factors through a morphism from $X\in \cu_{w\le 0}$ 
 to an element of $\cu_{w\ge 1}$ (by our assumptions on $D(E)$); hence $p$ factors as $i[1] \circ p'$, where $p' \in \cu(X,LE[1])$. 
Applying the octahedron axiom of triangulated categories to the  corresponding commutative triangle we obtain the existence of a distinguished triangle $T'\stackrel{e}{\to}T\to RE\to T'[1]$, where 
$LE\stackrel{b}{\to} T' \to X\stackrel{p'}{\to} LE[1]$ is a distinguished triangle corresponding to $p'$; moreover, 
 $e\circ b$ factors through $i$.
Since $RE\in \obj \du$, we can present $\phi = f\circ s^{-1}$ as $ (f\circ e)\circ (s\circ e)^{-1}$ in the category $\cu'$. 
Since the morphism $f\circ e\circ b$ factors through $i$,  it is zero by
our assumptions on $D(E)$ (since it factors through a morphism from $\cu_{w\le 0}$ into $Y\in \cu_{w\ge 1}$); hence 
 $f\circ e$ factors through $X$. Since $\cu(X,Y) = \ns$, 
 $f\circ e=0$ in $\cu$; 
  thus $\phi=0$.
 
Conversely, assume  that $w'$ is  a weight structure on $\cu'$. 
We take a $-1$-weight decomposition of an object $N$ of $\du\subset \cu$, i.e., a choice of the distinguished triangle (\ref{ewd}) for $M=N$ and $m=-1$. 
Then the "boundary" morphism $w_{\ge 0}N \stackrel{s}\to (w_{\le -1}N)[1]$ in this triangle    is invertible in $\cu'$ (since its cone is killed by $\pi$). 
Choose some $w_{\ge 1} N$ and  consider 
 the canonical morphism $w_{\ge 0} N \stackrel{f}\to w_{\ge 1} N$ "compatible with $g=\id_N$" as given by Proposition \ref{pbw}(\ref{icompl}) (cf. part \ref{ifilt} of that proposition). Then $f\circ s^{-1}=0$ 
in $\cu'$ by the orthogonality 
 axiom for $w'$.
Hence $\pi(f)=0$; thus $f$ factors through some object of $\du$. We take this object for $RN$ and complete the corresponding composition map $j:N\to w_{\ge 0} N\to RN$ to a distinguished triangle that we denote as in (\ref{ewwd}).  

It remains to verify that $i$ factors through $\cu_{w\le 0}$. 
Completing the commutative triangle $N\to RN\to w_{\ge 0}N$ to a commutative diagram
$$\begin{CD}
LN@>{i}>> N  @>{j}>> RN \\
@VV{}V                    @VV{\id_N}V                    @VV{}V \\
w_{\le 0}N@>{}>> N  @>{}>> w_{\ge 1} N\end{CD}
$$
we obtain that $i$ factors through $w_{\le 0}N\in \cu_{w\le 0}$ indeed. 

2. 
It certainly suffices to verify that the distinguished triangle (\ref{ewwdu}) satisfies the conditions for $D(N)$ (see (\ref{ewwd})), which is obvious. 

3. We should prove that the class  $D=(\obj \du \cap \cu_{w\le 0})\star_{\du} (\obj \du \cap \cu_{w\ge 0})$ (where $\star$ is the operation defined in Proposition \ref{pstar}(1))) equals $\obj \du$. 

In case (i) this class is extension-closed and $\al$-smashing according to Proposition \ref{pstar}(1,2). Since $B_0=\co(S)$ is a subclass of $\cu_{[0,1]}$ (see Proposition \ref{pbw}(\ref{i01})),  we have $B[i]\subset (\obj \du \cap \cu_{w\le 0})\cup  (\obj \du \cap \cu_{w\ge 0})\subset D$ for any $i\in \z$ (since the corresponding fact is valid for $B_j$ for $j=0,1,2$). Hence $D=\obj \du$ indeed. 

In case (ii) we note that $\lan(\cup_{i\in \z}\cp[i])\rab$ is the union of the classes $\lan(\cup_{i\in \z}\cp[i])\ral$ for all (infinite regular)  $\al$  (cf. 
the proof of Corollary \ref{classgws}(1)); hence $D=\obj\du$ in this case as well. 

Lastly, the functor $\pi$  respects the coproducts in question according to 
Proposition \ref{pneem}(2) (cf. also Corollary 3.2.11 of 
 \cite{neebook}); hence $w'$ is $\al$-smashing (resp., smashing) according to Proposition \ref{pdesc}(3). 

\end{proof}

\begin{rema}\label{rloc}
1. Part 3 of our theorem 
 generalizes  Theorem 4.2.3(1) of \cite{bososn}, where the case $B=B_0$ and $\al=\alz$ was considered. 
We also do not assume that $\cu$ is small (as was done in loc. cit.); yet this does not really affect the proof. 

In 
 \S\ref{ssheart} 
 below we will give  certain descriptions of  $\hw'$ in the settings 
of parts 2 and 3 of our theorem; these statements and their proofs are closely related to 
   Theorem 4.2.3(4--5) of \cite{bososn}. 

Note also that part 2 of that theorem is a particular case of our  Proposition \ref{pdesc}(2).

2. One of the main settings considered in ibid. was as follows: $\cu=K^b(\operatorname{FGProj}(R))$ is the homotopy category of bounded complexes (where $R$ is a unital ring and $\operatorname{FGProj}(R)$ is the category of finitely generated projective left $R$-modules), $w$ is the stupid weight structure  on $\cu$, $B=B_0=\co(S)$ 
 for $S$ being a set of morphisms in $\operatorname{FGProj}(R)$. 
  It was proved that the endomorphism ring of $\pi({}_RR)$ (${}_RR$ is $R$ considered as a left $R$-module) in $\cu'$ gives the so-called non-commutative or universal localization of $R$ with respect to $S$ (a-la Cohn; cf. \S\ref{sgenloc} below). 

3. Now we make some remarks related to part 1 of our theorem. 

 Applying Proposition \ref{pbw}(\ref{icompl}) we immediately obtain that a distinguished triangle of the form (\ref{ewwd}) satisfies the factorizability assumptions 
 in question if and only if $i$ factors through any fixed choice $w_{\le 0}N$ and $j$ factors through a (fixed) choice of $w_{\ge 0}N$.
Another equivalent formulation is that for any $Y \in \cu_{w\ge 1}, Y' \in \cu_{w\le -1}$, and morphisms $f\in \cu(N,Y),\ f'\in \cu(Y',N)$, the compositions $f\circ i$ and $j\circ f'$ are zero.

Next, we don't have any examples of weight-exact localizations $\pi:\cu\to \cu/\du$ for which we can prove that   weak $w$-decompositions inside $\du$ (as in part 2 of the theorem) do not exist. However, in \S\ref{sgenloc} we will describe a reasonable candidate for such an example in terms of so-called generalized universal localizations of rings.

Another interesting question is whether given any triangulated subcategory $\du$ of $\cu$ the class $\{N\in \obj \cu\}$ such that a triangle $D(N)$  (see (\ref{ewwd})) exists is extension-closed. 

\end{rema}

Some more comments to our theorem can be found in \S\ref{sex} below.

\subsection{Calculation of hearts for weight-exact localizations}\label{ssheart}
To calculate the functor $\hpi:\hw\to \hw'$ under the assumptions of Theorem \ref{tloc}(3) we prove that 
 it is "essentially determined" by the $w$-bounded subcategories of $\cu$ and $\du$.


\begin{pr}\label{pwede} 
Assume that inside a full strict triangulated subcategory $\du$ of $\cu$ there exist weak $w$-decompositions (in the sense of Theorem \ref{tloc}(2); here $w$ is a weight structure on $\cu$).

1. Then the full subcategories $\du_+$,  $\du_-$, and $\du_b$ of $\du$, whose object classes equal  $\obj \du\cap \obj \cu_+$, $\obj \du\cap \obj \cu_-$, and $\obj \du\cap \obj \cu_b$, respectively,  are triangulated.

2. Assume that the localization functor $\pi:\cu\to \cu'=\cu/\du$ exists. Then the restrictions of $\pi$ to   $\cu_+$,  $\cu_-$, and $\cu_b$, respectively, yield full embeddings into $\cu'$  of the Verdier localizations  $\cu_+/\du_+$,  $\cu_-/\du_-$, and $\cu_b/\du_b$, respectively. 

3. $\du_b$ is strongly generated (see \S\ref{snotata}) by $\obj\du \cap\cu_{[0,1]}$. Moreover, if $w$ restricts to $\du$ then $\du_b$ is strongly generated by $\obj\du \cap\cu_{w=0}$.


4. Denote by $S$ the morphism class $\{s\in \mo(\hw):\ \co(s)\in \obj\du\}$ and 
 take an infinite regular $\al$ such that $w$ is $\al$-smashing and $\du$ is an $\al$-smashing subcategory of $\cu$.

Next, for any set  
 $P\subset \cu_{w=0}$ we take $\cu^{P}=\lan P\ral$ and $\du^P=\lan \co(S\cap \mo(\cu^{P}))\ral$. Then these categories are essentially small, $w$ restricts to a weight structure $w^P$ on $\cu^P$, and $w^P$ descends to a weight structure $w'{}^P$ on $\cu'{}^P=\cu^P/\du^P$. Moreover, we have a natural direct system of weight-exact functors $i^P:\cu'{}^P\to \cu'$ (for $P$ running through all subsets of $\cu_{w=0}$) that gives 
a full embedding of the $2$-colimit category $\inli_P \cu'{}^{P}_{b}$ into $ \cu'_b$.


\end{pr}
\begin{proof}

1. Obvious (recall that $\cu_+$,  $\cu_-$, and $\cu_b$ are full strict triangulated subcategories of $\cu$).

2. $\pi$ restricts to a full embedding  $\cu_-/\du_-\to \cu'$ according to (the easy) Proposition \ref{plocsub} below. 
 Dualizing this statement we obtain that $\cu_+/\du_+$ embeds into $\cu'$ as well.  Lastly, combining these two statements one obtains that $\pi$ embeds the category $\cu_b/\du_b$ into $\cu'$ as well (recall that $w$ restricts to $\cu_-$ and note that the objects of $\cu_-$ that are bounded below   with respect to this restriction are precisely the objects of $\cu_b$).

3.  This statement can be easily deduced from the simple  Proposition 1.1.9 of \cite{bbd}, and it essentially coincides with Corollary 3.1.4(2) of \cite{bsosnl}.

4.  The category $\cu^{P}$ is essentially small according to Proposition 3.2.5 of \cite{neebook}; hence $\du^{P}$ also is. $w$ restricts to $\cu^P$ according to Theorem \ref{tvneg}(\ref{itvn1},\ref{itvn4}) (along with Proposition \ref{prestr}(1); see also Remark \ref{rcompar}(\ref{ialz}) for the case $\al=\alz$). $w^P$ descends to $\cu'{}^P$ according to Theorem \ref{tloc}(3(i)). Our definitions easily give the existence and weight-exactness of the functor $i^P$; see the 
diagram $$\begin{CD}
 \du^P@>{}>>\cu^P@>{\pi^P}>>\cu'{}^P\\
@VV{}V@VV{}V@VV{i^P}V \\
\du@>{}>>\cu@>{\pi}>>\cu'
\end{CD}$$

Next, applying assertion 2 we obtain that in the commutative square of functors 
$$\begin{CD}
 \cu^{P}_b/\du^{P}_b@>{}>>\cu'{}^P\\
@VV{}V@VV{i^P}V \\
\cu_b/\du_b@>{}>>\cu'
\end{CD}$$
the 
 rows are full embeddings.
 Recalling the definition of Verdier localizations along with our assertion 3 (and Proposition \ref{pbw}(\ref{i01})) we obtain that 
 it remains to verify the following: for any set $Q$ of objects of $\cu_b$ there exists a  {\bf set} $P\subset \cu_{w=0}$ such that $Q\subset \obj \cu^P$. The latter statement immediately follows from Proposition \ref{pbw}(\ref{iextcub}).

\end{proof}

Now we 
 are able (essentially) to calculate the functor $\hpi$. Our arguments are  rather similar to that of \cite{bososn} (and also rely on several results of loc. cit.); yet the conclusions are much more general.

\begin{theo}\label{theart}
Assume that 
a weight structure $w$ on $\cu$ is $\al$-smashing for some regular $\al$ (resp., $w$ is smashing); 
  let   $\du$ be the $\al$-localizing (resp. localizing) subcategory of $\cu$  
 generated by $B=B_0\cup B_1 \cup B_2$, where $B_0$ consists of cones (in $\cu$) of a class $S_0$ of $\hw$-morphisms, elements of $B_1$ are left $w$-degenerate, and elements of $B_2$ are right $w$-degenerate. 
  Suppose that the localization functor $\pi:\cu\to \cu'=\cu/\du$ exists (in the sense of Remark \ref{rclass}(3)).

1. Then for the heart $\hw'$ of the the weight structure $w'$ descended from $w$ to $\cu'$ (see Theorem \ref{tloc}(3)) 
 the corresponding functor $\hpi:\hw\to \hw'$ factors as the composition of the localization $\pi_{\hw}$ of the category $\hw$ by the morphism class $S=\{s\in \mo(\hw):\ \co(s)\in \obj\du\}$ 
  with a full embedding $\hw[S\ob]\to \hw'$.

2. Any morphism in this category $\hw[S\ob]$ can be presented as $gs\ob i$, where $g$, $s$, and $i$ are $\hw$-morphisms, $s\in S$, and $i$ is a coretraction (i.e., it is split injective). 

3. $\hw[S\ob]$ is the target of the universal additive functor $\hw\to \bu $ that makes all elements of $S_0$ invertible and respects coproducts of less than $\al$ objects (resp., all small coproducts; so, a functor satisfying this universality condition exists whenever $\cu'$ does).

4. $\hw[S\ob]$ is isomorphic (in the obvious way) to the category $\hw[S'{}\ob]$, where $S'$ is the closure of $S_0\cup \{\id(M):\ M\in \cu_{w=0}\}$ with respect to coproducts of less than $\al$ morphisms (resp., with respect to all small coproducts).

\end{theo}
\begin{proof}

It is easily seen that Proposition \ref{pwede}(4) reduces the $\al$-smashing version of our assertions to the case  where $\cu$ is small.
Moreover, this proposition also implies the following: if $\cu$, $w$, and $\du$ are as in the smashing version of our assertions, 
 and for any regular cardinal $\be$ and $P$ being a subset of $\cu_{w=0}$ we write $\cu^{P,\be}$ for $\lan P\ra_{\cu}^{\be}$, $\du^{P,\be}=\lan \co(S\cap \mo(\cu^{P}))\ra^{\be}$, then for the corresponding weight structures $w'{}^{P,\be}$ on $\cu'{}^{P,\be}=\cu^{P,\be}/\du^{P,\be}$ the natural functors $i^P:\cu'{}^{P,\be}\to \cu'$ induce a full embedding of the $2$-colimit $\cu'{}^{P,\be}_b$ into $ \cu'_{b}$ (here we set $(P,\be)\le (P',\be')$ whenever $P\subset P'$ and $\be\le \be'$). Hence to prove this version of  our theorem it also  suffices to consider the case where $\cu$ is small and $\al$-smashing (for some regular $\al$). 
 Thus we can  ignore the existence of  localizations 
 matter.

1. Proposition \ref{pwede} allows us to assume that $w$ is bounded, $\du$ 
 is strongly generated by 
 the set  $\co(S)$, and (respectively) $\al=\alz$. In this case Theorem 4.2.3(4) of \cite{bososn} says that 
the functor $\hpi$ factors as the composition of a certain localization functor $\hw\to \hw[S\ob]_{add}$ with a full embedding. It remains to note that this functor $\hw\to \hw[S\ob]_{add}$ is isomorphic to the localization 
 $\pi_{\hw}: \hw\to \hw[S\ob]$ according to Corollary 2.2.3 of ibid.


2. Immediate from  Proposition 3.1.2 of ibid. (here we use the aforementioned isomorphism $\hw[S\ob] \cong \hw[S\ob]_{add}$; cf.  also Theorem 0.1 of ibid. for the definition of $\hw[S\ob]_{add}$). 

3. Since $\pi$ respects coproducts of less than $\al$ objects (see Proposition \ref{pneem}(2)) and  $\hw$ is closed with respect to $\cu$-coproducts of this sort (see Proposition \ref{pbw}(\ref{ismash})),  the 
localization functor $\pi_{\hw}$   is an additive functor that respects coproducts of less than $\al$ objects and makes all elements of $S_0$ invertible. Thus to check our assertion it suffices to verify that any functor $F:\hw\to \bu$ that fulfils these properties uniquely factors through  $\pi_{\hw}$. 
Moreover, assertion 2 implies that the factorization is unique if exists.

Now the theory of weight complexes (that we will recall below) gives an explicit factorization of $F$ through $\pi_{\hw}$. 

We consider the composition $G=K_\w(F)\circ t:\cu\to K_\w(\bu)$ (see Definition \ref{dkw} and Proposition \ref{pwc}).
We claim that $G$ factors through the localization functor $\pi:\cu\to \cu'$. 
By the universal property of ``abstract'' localizations, it suffices to verify for any $\cu$-distinguished triangle $X\to Y\stackrel{f}{\to} Z\to X[1]$ that $G(f)$
is an isomorphism if $X\in \obj \du$.
According to Proposition \ref{pwc}(\ref{iwct}), 
 it suffices to check
that $G(X)=0$. 
Moreover, Proposition \ref{pwc}(\ref{iwcf},\ref{iwcons},\ref{iwct}) confines us to checking that $G(X)=0$ if $X\in \bu$. Now, the latter fact is given by Proposition \ref{pwc}(\ref{iwcern}) if $X\in B_1$ or $X\in B_2$, whereas for $X\in B_0$ one should 
 apply parts \ref{iwcalc} and \ref{iwct} of the proposition.

Thus the restriction of $G$ to $\hw$ factors through $\pi_{\hw}$, and it remains to note that this restriction equals the composition of $F$ with the embedding $\hw\to K_\w(\hw)$ 
 according to part \ref{iwcalc} of the proposition.

4. The functor $\hw\to \hw[S'{}\ob]$ is additive and respects coproducts of less than $\al$ objects according to (the rather easy) Proposition \ref{addloc}  below. Hence it factors through $\pi_{\hw}$ according to the universality property given by the previous assertion. Since any additive functor from $\hw$ that respects coproducts of this sort and makes elements of $S_0$ invertible also makes all elements of $S'$ invertible, we obtain that $\pi_{\hw}$ factors through this localization $\hw\to \hw[S'^{-1}]$. Hence these two localization functors are isomorphic.

\end{proof}  

\begin{rema}\label{rheart}
\begin{enumerate}
\item\label{irh1} Certainly, 
 for $(\cu,w,\cu',w')$ as in  the theorem the category $\hw'$ 
 equals $\kar_{\cu'}(\hw[S\ob])$ (see Proposition \ref{pdesc}(2)).  Moreover, in the case $\al>\alz$ and in the smashing one the category $\hw'$ is Karoubian 
(since $\cu'$ is so; see  Remark \ref{rkar}(3)); 
hence   $\hw'$ is equivalent to  $\kar(\hw[S\ob])$. 

Moreover, if $\al>\alz$ then  the category $\inli_P \cu'{}^{P}_b$ is easily seen to be equivalent to  $ \cu'_b$ under the assumptions of Proposition \ref{pwede}(4).

\item\label{irh2}
Recall also that Propositions 3.2.2.2 and  3.1.7 of \cite{bososn} (see also the text preceding the latter proposition) give a certain "explicit" description of the category $\hw[S'{}\ob]$. So, we have an explicit description of all morphisms in this localization (that is closely related to part 2 of our theorem) as well as of a sum, a direct sum, and any composition of morphisms presented in this form; the most nasty part of that proposition gives a criterion for two parallel morphisms to be equal.\footnote{This description is closely related to the main results of \cite{gera} and \cite{malc} on so-called non-commutative localizations of rings; see \cite[\S3.2.2]{bososn}.}

Moreover, the description of $\hw[S'{}\ob]$ (and so, also of $\hw'$) becomes much easier if $B_0\subset \cu_{w=0}$ (localizations of this type can be called {\it pure} ones). We will discuss this case and some examples for it in Remark \ref{rrheart} below.

\item\label{irh4}
For an $\al$-smashing additive Karoubian (say, small) category $\au$ one can take $w$ to be the stupid weight structure on the category $\cu=K(\au)$. Then $\au \cong \hw$ and for any $S_0\in \mo(\au)$ 
part 4 of our theorem says that the corresponding functor $\au\to \au[S\ob]$ is a restriction of the localization $\pi:\cu\to \cu'$.
 Thus one obtains a "triangulated method" of calculating the universal functor  described in part 3 of our theorem.

 One can also easily prove that it is actually not necessary to assume that $\au$ is Karoubian in this statement (cf. the proof of  \cite[Theorem 4.2.3(5)]{bososn}).


\item\label{irhcl}
The reduction to the case of a small $\cu$ in the beginning of the proof of our theorem is significantly harder than similar arguments in other 
 parts of the paper. So we describe certain alternatives to this reasoning. Firstly, it appears that the corresponding results of \cite{bososn} (and their proofs) can be generalized to the setting of not necessarily small categories. Secondly, one can certainly "pass to a larger universe". Lastly, instead of Proposition \ref{pwede}(4) one may apply Proposition \ref{prestrc} below (see Remark  \ref{restrc}(\ref{irest2})).

\end{enumerate}
\end{rema}

\subsection{Relation to generalized universal localizations of rings}\label{sgenloc}


We recall Definition 7.1 of \cite{bapa} (cf. also \S15 of \cite{krahomloc} where  generalized universal localization of rings were called  homological localizations).

\begin{defi}
Let $R$ be a ring and $B$ a set of compact objects of $D(R)$. 
Then the {\it generalized universal localization} of $R$ with 
respect to $B$ 
is the initial object in the category of ring homomorphisms $R\to U$ such that $T \otimes^{\mathbb{L}}_R U \cong 0$ for any $T \in B$.
\end{defi}

This notion generalizes Cohn's universal localization as the latter is the universal ring homomorphism such that $\co(s) \otimes^{\mathbb{L}}_R U \cong 0$ for $s$ running through a set of morphisms between finite dimensional projective modules. 

Unlike the Cohn's universal localization, the 
generalized universal localization doesn't always exist (see Example 15.2 in \cite{krahomloc}).

The relation of our constructions to the theory of generalized universal  localizations is given by the following proposition. 

\begin{pr}\label{lochomloc}
Let $R$ be a ring and $B$ a set of objects of $\cu=K^b(\operatorname{FGProj}(R))$ 
such that the stupid weight structure on 
$\cu$  (see Remarks \ref{rstws}(1) and \ref{rloc}(2))  descends to the localization $\cu'$ of $\cu$ by the triangulated subcategory $\du$ strongly generated by $B$. 

Then the generalized universal localization with respect to $B$ exists and is isomorphic to the endomorphisms ring of the image $
\pi({}_RR)\in \obj \cu'$ (${}_RR$ is $R$ considered as a left $R$-module) in this localization. 
\end{pr}
\begin{proof}
Any map of rings $f:R \to U$ induces an exact  functor 
$- \otimes^{\mathbb{L}}_R U: K^b(\operatorname{FGProj}(R)) \to K^b(\operatorname{FGProj}(U))$. 

The assumptions of the theorem 
 say that there is a weight structure $w'$ on 
 $\cu'$ and the localization functor $\pi:
\cu\to \cu'$ is weight-exact. 

Now 
assume that $f$ 
makes objects of $B$ acyclic. By definition of localization, the functor $
\cu \to K^b(\operatorname{FGProj}(U))$ factors through $\pi$. 
By  Corollary 3.5 of \cite{sosnwc} (cf. also \S6.3 of \cite{bws}; note that all these categories are easily seen to admit {\it differential graded enhancements}, see  ibid.) there exists a weight-exact functor $t_{\cu'}:
 \cu' \to K^b(\hw')$ (this is the "strong version" of the weak weight complex functor; see \S\ref{swc} below) and  
moreover the functor $
\cu \to K^b(\operatorname{FGProj(U)})$ factorizes as 
$$
\cu \stackrel
{t_{\cu'}\circ \pi}\to K^b(\hw')\stackrel{q}\to K^b(\operatorname{FGProj}(U))$$ 
uniquely. 
This implies that $f$ factors through the induced map $R \to \operatorname{End}_{K^b(\hw')}(\pi({}_RR))$. 

Moreover, this factorization is unique. Indeed, any unital ring homomorphism $g:\operatorname{End}_{K^b(\hw')}(\pi({}_RR))\to U$ gives rise to a functor $q_g:K^b(\hw') \to K^b(\operatorname{FGProj}(U))$; since $q_g=q$ as we have just proved, we immediately obtain the uniqueness of $g$. 
\end{proof}


Applying Theorem \ref{tloc}(1) along with Remark \ref{rloc}(3) (and Remark \ref{rstws}(1)) we immediately deduce the following statement.

\begin{coro}
Let $R$ be a ring  and $B$ a set of objects of $\cu=K^b(\operatorname{FGProj}(R))$. 
Assume that for any object $N$ of the triangulated subcategory $\du\subset \cu$ strongly generated by $B$ there exists a $\du$-distinguished triangle 
$LN \stackrel{i}{\to} N\stackrel{j}{\to} RN\to LN[1]$ and a choice of stupid truncations $N^{\ge 0}$ and $N^{\le 0}$ (concentrated in cohomological degrees at least zero and at most zero, respectively)
such that $i$ (resp. $j$) factors through the corresponding morphism  $N^{\ge 0}\to N$ (resp. $N\to N^{\le 0}$). 
 

Then the generalized universal localization of $R$ with respect to $B$ exists. 
\end{coro}

\begin{rema}
We don't know whether this condition is necessary. If the 
generalized universal localization of $R$ with respect to class of complexes 
$B$ exists, we do have a sequence of functors 
$$K^b(\operatorname{FGProj}(R)) \to K^b(\operatorname{FGProj}(R))/\du \to K^b(\operatorname{FGProj}(U))$$
The second functor is not an equivalence in general and we don't know whether  there necessarily exists a  descended weight structure on $K^b(\operatorname{FGProj}(R))/\du$.
\end{rema}

However, 
 under certain assumptions we are able to prove that the condition is necessary using the results of \cite{bapa}. 

\begin{pr}\label{homephomloc}
Let $R$ be a ring; 
let $T$ 
 be a  {\it classical partial $n$-tilting (left) $R$-module}\footnote{I.e., $T$ is an $R$-module with  $\du(R)(T,T[i])=0$ for all $i\neq 0$ that admits a projective   finitely generated resolution of length $n\ge 0$ (see also Definition 5.1 of \cite{bapa}).}. 
Assume that the embedding $\langle T \rangle^{\operatorname{cl}} \to D(R)$ admits a left adjoint $\pi$ and  
 $H^i(\pi({}_RR)) = 0$ 
for all $i\neq 0$. 

Then the following holds:

1. 
 The generalized universal localization $R\stackrel{f}\to U$ of $R$ with respect to the set $\{T\}$ exists. 

2. 
Denote by $\du$ the full subcategory of $K^b(\operatorname{FGProj}(R))$ strongly generated 
by $T$. 
Then the functor $K^b(\operatorname{FGProj}(R))/\du \to K^b(\operatorname{FGProj}(U))$ is fully faithful. 
\end{pr}
\begin{proof}
The first statement is precisely Proposition 7.3 of \cite{bapa}. 

The second statement follows from the equivalence between condition 1 and 2 in Theorem 6.1 of \cite{bapa}. Indeed, condition 2 of loc. cit. implies that the 
functor $- \otimes^{\mathbb{L}}_R U$ induces an equivalence 
$D(R)/\langle T \rangle^{\operatorname{cl}} \cong D(U)$. Hence Proposition 4.4.1 of \cite{neebook} yields the result. 
\end{proof}

\begin{rema}
In Example 8.1 of  \cite{bapa} 
 a  certain weight-exact localization 
 functor was constructed.
The authors take an explicit classical partial tilting module $T$  over a certain (noetherian) ring $R$ that has projective dimension $2$, and apply Proposition \ref{homephomloc} to show that the  generalized universal localization of $R$ with respect to 
  (a projective resolution) of $T$ exists. 
Now,  
 this proposition implies that the stupid weight structure on   $K^b(\operatorname{FGProj}(R))$ descends  to its localization by the subcategory strongly generated by $\{T\}$. 

The authors  suspect that the corresponding subcategory $\du$ does not admit ("stupid") weak weight decompositions; yet they don't know how to prove this. One can probably deduce this statement from Theorem \ref{theart} (cf. also Remark \ref{rheart}(\ref{irh2})).
 Another interesting question is to find 
 "explicit" decompositions of the type (\ref{ewwd}) for objects of $\du$.  
\end{rema}

\section{Supplements, examples, and applications}\label{suppl}

In this section we recall some definitions and statements that were used in the proofs above, obtain some new results, and discuss examples to the propositions  of previous  sections.

In \S\ref{salmb} we construct a weight-exact right adjoint $\rinf$ to the embedding $\cu\abo\to \cu$ (see Remark \ref{rabo}; our construction uses {\it countable homotopy colimits} as introduced in \cite{bokne}). Next we use it to prove Theorem \ref{tvneg}(\ref{itvn6}) and to prove that $\cu$ "often" equals $\lan \cu_{w=0}\ralo$. Moreover, we study the conditions ensuring that a countably  smashing weight structure $w$ on $\cu$ restricts to an $\alo$-localizing subcategory $\du\subset \cu$.

In \S\ref{sex} we give some more comments and examples to the results of  the previous section. 

In \S\ref{swc} we recall some basics of the theory of weight complexes (that was used in the proof of Theorem \ref{theart}(3)). We also combine this theory with the main results of previous sections to obtain certain "weight complex criteria" for an object $M$ of $\cu$ to belong to a triangulated subcategory of $\cu$ (that is generated by certain "$w$-simple" objects), and ensuring that $M$ belongs to a certain $\du\subset \cu$ if it belongs to $\obj \cu_1\cap\obj \cu_2$.

In \S\ref{slloc}  we prove two simple 
 lemmas on localizations of additive categories.

\subsection{Almost bounded above 
objects as homotopy colimits} 
\label{salmb}

Now we prove a collection of nice properties of almost bounded above objects; 
this gives a proof of Theorem \ref{tvneg}(\ref{itvn6}). Moreover, we will prove that  
$w$ can be easily "recovered from" its restriction to $\cu\abo$ (see Remark \ref{rabo});  
recall that $\cu\abo$ is the $\alo$-localizing subcategory of $\cu$ generated by $\cu_-$).
So we assume that $\cu$ is $\alo$-smashing. 
Our arguments are based on 
{\it countable homotopy colimits} (in triangulated categories) as defined in \cite{bokne}. So we start from recalling some basics on this notion.

\begin{defi}\label{dcoulim}

For a sequence of objects $Y_i$ 
 of $\cu$ for $i\ge 0$ and maps $\phi_i:Y_{i}\to Y_{i+1}$  
 we consider  $D=\coprod Y_i$ and the morphism $a:\oplus \id_{Y_i}\bigoplus \oplus (-\phi_i): D\to D$; 
 denote a cone of $a$ by $Y$. We will write $Y=\hcl Y_i$ and call $Y$ a {\it homotopy colimit} of $Y_i$. 
\end{defi}

The following properties of homotopy colimits (of this sort) are probably well-known; all of them were justified in \S4.1 of \cite{bpure} (cf. also  \S3.5 of \cite{weibook}). 

\begin{rema}\label{rcoulim}
1. These homotopy colimits are not really canonical and functorial in $Y_i$ since the choice of a cone is not canonical. They are only defined up to non-canonical isomorphisms; still this is satisfactory for our purposes.


2.  Let $H$ be a cohomological functor from $\cu$ into abelian groups that converts countable coproducts into products. Then 
the easy Lemma 4.1.3(1,2) of \cite{bpure} essentially gives the following (probably, well-known) statement:
 the  connecting morphisms $Y_i\to Y$  give an epimorphism  $H(Y)\to \prli H(Y_i)$, and it is an isomorphism whenever the maps $H(\phi_i[1])$ are surjective for $i\gg 0$.

In particular, for any $C\in\obj\cu$ 
 the  connecting morphisms give a  surjection $\cu(Y,C)\to \prli \cu(Y_i,C)$. Moreover, if for a compatible system of morphisms $f_i:Y_i\to C$ we consider  any $f:Y\to C$ that "lifts" $(f_i)$ via this surjection then for any $H$ satisfying the aforementioned conditions the map $H(f)$ is isomorphic to $\prli H(f_i)$.

3. Furthermore, if $H':\cu\to \ab$ is a homological functor that respects countable coproducts  then we have   $\inli H'(Y_i)\cong H(Y)$ (see part 3 of the aforementioned lemma). Yet we will not apply this statement below.  
\end{rema}

\begin{theo}\label{tabo}
Assume that $w$ is an $\alo$-smashing weight structure (on $\cu$), and $\al$ is some infinite regular cardinal.
Then the following statements are valid.

\begin{enumerate}

\item\label{iabdj} The embedding $\cu\abo\to \cu$ possesses 
 an exact right adjoint $\rinf$ 
 whose kernel is the subcategory $\cu_{+\infty}$ (recall that $\cu_{+\infty}$ is the subcategory of left $w$-degenerate objects), whereas the embedding $\cu_{+\infty}\to \cu$ possesses an exact left adjoint $\linf$. Moreover, for any $M\in \obj \cu$ there exists a distinguished triangle $\rinf(M)\to M\to \linf(M)\to \rinf(M)[1]$, and it is essentially the unique triangle $LM\to M \to RM\to LM[1]$ such that $\rinf(M)\in \obj \cu\abo$ and $\rinf(M)\in \obj \cu_{+\infty}$. 

\item\label{iahcl}  For any $M\in \obj \cu$ the object $\rinf(M) $ is $\cu$-isomorphic to $\hcl_i w_{\le i}M$; here we take arbitrary choices of  $w_{\le i}M$ for all $i\ge 0$ and connect these objects by the unique morphisms provided by Proposition \ref{pbw}(\ref{ifilt}).

\item\label{iaort}
$\obj \cu \abo=\perpp \obj \cu_{+\infty}$; hence $\obj \cu \abo$ is $\al$-smashing  whenever $\cu$ is.

\item\label{iabext} For the restriction $w\abo$   of $w$ to $\cu\abo$  (see Remark \ref{rabo}(1))   we have $\cu\abo_{w\abo\le 0}=\cu_{w\le 0}$, whereas $\cu_{w\ge 0}$  consists precisely of extensions of elements of $\cu\abo_{w\abo\ge 0}$ by objects of  $\cu_{+\infty}$.

Moreover, the functor $\rinf$ is weight-exact.

\item\label{iabdhcl}
 $\cu\abo_{w\abo\ge 0}$ coincides with the class of all $\hcl Y_i$ with $Y_i\in \obj \cu_-\cap \cu_{w\ge 0}$.

\item\label{iarc} Assume that the class $\obj \cu_{+\infty}$ is $\al$-smashing in $\cu$ 
(in particular, this is the case if $w$ is $\al$-smashing). Then 
 $\rinf$ respects coproducts of less than $\al$ objects.
\end{enumerate}

\end{theo}
\begin{proof}  For $M'=\hcl_i w_{\le i}M$ we   choose some connecting morphism  $c:M'\to M$ compatible with the $i$-weight decomposition morphisms $w_{\le i}M\to M$ (see Remark \ref{rcoulim}(2)).

We claim that a cone $C$ of $c$ is a left $w$-degenerate object of $\cu$. To verify this statement we need certain cohomological functors. So for any $N\in \obj \cu$ we consider the {\it virtual $t$-truncations} of the representable functor $H_N=\cu(-,N)$; in \S2.4 of \cite{bpure} these truncations were denoted by $\tau^{\ge  -n}H_N$ for $n\in \z$, whereas in \S2.1 of \cite{bvtr} the notation $\tau_{\le  n}H_N$ was used. The latter functors are cohomological functors from $\cu$ into $\ab$  (see \S2.5 of \cite{bws}) that send countable coproducts into products (see 
 Remark 2.5.2 of \cite{bpure} or Remark 3.1.3(1) of \cite{bvtr}; recall that $w$ is $\alo$-smashing). Moreover, Proposition 2.1.6 of ibid. (as well as Proposition 2.4.6 of  \cite{bpure})  immediately implies that for $T\in \obj \cu$ and $m\in \z$ we have $T\in \cu_{w\ge m}$ if and only if  $\tau_{\le  l}H_N(T)=\ns$ for all $N\in \obj \cu$ and $l<m$ (we will now apply both implications of this criterion).

Hence the application of  all  $\tau_{\le  l}H_N$ to the connecting morphisms  $ w_{\le i}M\stackrel{c_i}{\to} w_{\le i+1}M\to   M$ 
gives isomorphisms whenever $i<l-1$ (apply  $\tau_{\le  l}H_N$ to the corresponding cones). Thus the homomorphisms $\tau_{\le  l}H_N(c)$ 
 are bijective for all $N$ and $l$; therefore $C$ is left $w$-degenerate indeed. %
Moreover, the definition of $M'$ implies that $M'\in \obj \cu\abo$.

Now recall that  $ \obj \cu\abo\perp \obj \cu_{+\infty}$ according to Remark \ref{rabo}(2). Hence applying  Theorem 9.1.13 of \cite{neebook} we obtain the following: the existence of distinguished triangles $M'\to M\to C\to M'[1]$ with $M'\in \obj \cu\abo$ and $C\in   \obj \cu_{+\infty}$ imply that  {\it a Bousfield localisation functor exists} for the pair  $  \cu\abo\subset \cu$ (see Definition 9.1.1 of ibid.). 
Combining this statement with other results of ibid. (see the proof of the aforementioned theorem and Proposition 9.1.18 of ibid.) we easily obtain 
assertions \ref{iabdj}--\ref{iaort} (note also that the functors $\rinf$ and $\linf$ are exact according to  Lemma 5.3.6 of ibid.). 

\ref{iabext}. $\cu\abo{}_{w\abo\le 0}=\cu_{w\le 0}$ since $w\abo$ is a restriction of $w$ and $\obj \cu \abo$ contains $\cu_{w\le 0}$. Hence both $w$ and $w\abo$ are generated by $\cu_{w\le 0}$ (in the corresponding categories). Hence 
Propositions  2.6(II.3) and 3.2(5) of \cite{bvt} 
  give the desired relationship of $\cu\abo{}_{w\abo\ge 0}$ to $\cu_{w\ge 0}$.

If $M\in \cu_{w\le 0}$ then we certainly have $\hcl w_{\le i}M=M$; hence $\rinf$ is left weight-exact.  Lastly, the description of $\rinf$ reduces the  right weight-exactness of $\rinf$ to  assertion \ref{iabdhcl}. 


\ref{iabdhcl}.  Since the class $\cu\abo_{w\abo\ge 0}$ contains  $\obj \cu_-\cap \cu_{w\ge 0}$, contains countable $\cu$-coproducts of its objects, and the shift $[1]$ restricts to it, the definition of $\hcl Y_i$ implies that  $\hcl Y_i\in \cu\abo_{w\abo\ge 0}$  whenever $Y_i\in \obj \cu_-\cap \cu_{w\ge 0}$.

Conversely, 
assertion \ref{iahcl} implies that any $Y\in \cu\abo_{w\abo\ge 0}$ can be presented as $\hcl w_{\le i}Y$, and it remains to note that  $w_{\le i}Y\in \cu_{[0,i]}$ (see Proposition \ref{pbw}(\ref{iwd0})).

\ref{iarc}. This is also an easy consequence of the theory in \cite[\S9.1]{neebook}; see Remark 1.3.5(4) and Proposition 1.3.4(4) of \cite{bpure} or Proposition 3.4(5) of \cite{bvt} for more detail. 
\end{proof}

These results have some important consequences.

\begin{coro}\label{cabo}
I. Assume that $\cu$ is  $\al$-smashing for some (regular cardinal) $\al\ge \alo$. 

1. Suppose that $w'$ is an $\alo$-smashing weight structure on an $\alo$-localizing subcategory $\cu'$ of $\cu$, and the class $\cu'_{w'=0}$ is $\al$-smashing in $\cu$.
 Then  the classes $\cu'_{[m,n]}$ (for all $m,n\in \z$) and    $\cu'\abo_{w'\abo\ge 0}$ are $\al$-smashing in $\cu$ as well.

2. A class $\cp\subset  \obj \cu$ is $\al$-negative 
if and only if the class $\cp\tal$ consisting of all coproducts of families of elements of $\cp$ of cardinality less than $\al$ is countably negative.

II. Assume that $\cu$ is smashing.

1. Then a  class $\cp\subset  \obj \cu$ is class-negative if and only if the class $\cp\tbig$ consisting of all small coproducts of  elements of $\cp$  is countably negative.

2. Assume that $\cu$ 
is endowed with a $\alo$-smashing 
weight structure $w$, the class $\obj \cu_{+\infty}\subset \obj \cu$ is smashing,  and $\cu=\lan \cp \ra\tbig$, where 
$\cp\subset \obj(\lan \obj \cu_+ \ralo)$ is a set that is  countably perfect (i.e.  the class $\operatorname{Null}(\cp)= \{f\in \mo(\cu);\ \cu(P,-)(f)=0\ \forall P\in \cp\}$ is closed with respect to countable coproducts; see Remark \ref{rcgws}(\ref{iperf})). Then $\cu \abo   =\lan \cu_{w=0} \ralo$.

\end{coro}
\begin{proof}

I.1.  Certainly, the class $\cu'_{[j,j]}$ is $\al$-smashing in $\cu$  for any $j\in \z$. Then easy induction on 
 $n-m$ yields that $\cu'_{[m,n]}$ is $\al$-smashing in $\cu$  for any $m,n\in \z$; here one should combine the distinguished triangles provided by Proposition \ref{pbw}(\ref{ifilt}) with Proposition \ref{pneem}(1).\footnote{And it is not necessary to assume that $w'$ is $\alo$-smashing for this part of the argument.} Thus the class $\obj\cu'_-\cap \cu'_{w'\ge 0}=\cup_{i\ge 0}\cu'_{[0,i]}$ is $\al$-smashing also,  
  and applying Theorem \ref{tabo}(\ref{iabdhcl}) we obtain that $\cu'\abo{}_{w'\abo\le 0}$ is $\al$-smashing as well (here we use the obvious fact that $\cu$-homotopy colimits respect $\cu$-coproducts). 

2. 
The "only if" implication is obvious. 
 
Conversely, assume that $\cp\tal$ is countably negative. Denote by $w'$  the weight structure $\alo$-generated by $\cp^{\al}$ on the category $\cu'=\lan \cp \ralo$ (see Theorem \ref{tvneg}(\ref{itvn1})); note that all objects of $\cu'$ are almost $w'$-bounded above by the definition of this notion. 
The previous assertion  implies that the class $\cu'_{w'\ge 0}$ is $\al$-smashing in $\cu$. Hence $\cu'_{w'\ge 0}=[ \cup_{i\ge 0}\cp[i] \bral$ 
 and the orthogonality axiom for $w'$ implies that $\cp\perp [ \cup_{i\ge 1}\cp[i] \bral$ as desired.

II.1. Obvious from assertion I.2 (recall that $[(\cup_{i\ge 1}\cp[i])\brab$ is the union of $[(\cup_{i\in I}\cp[i])\bral$ for all infinite regular  $\al$). 

2. The corollary in the end of \cite[\S1]{kraucoh}\footnote{See also Corollary 2.6 of \cite{modoicogen}; moreover, one may  apply 
Theorem 4.3.1 of \cite{bpure}  to the countably perfect set $\cup_{i\in \z}\cp[i]$ (cf. Definition 4.1.4 of ibid. and recall that $\lan \cu_{w=0} \ralo$ is Karoubi-closed in $\cu$ according to 
Remark \ref{rkar}(3)).} 
implies that $\cu=\lan \cp\tbig\ralo$. Next, Theorem \ref{tabo}(\ref{iarc}) certainly yields that the functor $\rinf$ respects all small coproducts. Thus $\cu\abo=\lan \rinf(\cp)\tbig\ralo$ 
 (note that $\rinf$ is essentially surjective on objects). Lastly recall that $\rinf$ is also  weight-exact (see part \ref{iabext} of the theorem); hence $\rinf(\cp)\subset \lan \obj \cu\abo{}_+ \ralo=\lan\obj \cu_{b} \ralo =\lan \cu_{w=0} \ralo$ (here one can apply part \ref{iabdhcl} of our theorem to obtain the first of these equalities).

\end{proof}

\begin{rema}\label{rconj}
1. 
Obviously, if $w$ is as in part II.2 of our corollary and also left non-degenerate then 
 $\cu=\lan \cu_{w=0}\ralo$; hence $w$ is $\alo$-generated by $\cu_{w=0}$.
Moreover, even if $w$ is left degenerate we may still "recover" $w$ from the weight structure $\alo$-generated by  $\cu_{w=0}$; see Theorem \ref{tabo}(\ref{iabext}).

2. If $w$ is ($\alo$-smashing and) left non-degenerate then 
 Theorem \ref{tabo}
(\ref{iabext})  implies that $\cu_{w\ge 0}$ is contained in $\cu\abo$. In particular, this is the case for the stupid weight structure of $K(\bu)$ for any $\al$-smashing additive $\bu$ (see Remark \ref{rstws}; certainly, in this case this statement can be easily verified directly).

 3. It may make sense to prove that $\co(c)\in \cu_{w\ge i}$ (in the proof above) for all $i\in \z$ via computing $\cu(Z,-)(M')$ for $Z\in \cu_{w\le i}$.  The problem is that these functors are "difficult to control" if $Z$ is not compact; so this  plan is not so easy to realize if $w$ is not generated 
 by a class of compact objects. Still the authors have some ideas for overcoming this difficulty (and considering arbitrary  $Z\in \cu_{w\le i}$); this requires some new lemmas on countable homotopy colimits in $\cu$.

This method may also yield some alternative to Corollary \ref{cabo}(II.2) that would ensure that $\cu\abo=\lan \cu_{w=0}\ralo$ under certain assumptions on the (existence and properties of) the $t$-structure {\it right adjacent} to $w$; see 
\S1.3 and Theorem 3.2.3 of  \cite{bvtr}.
Yet 
 this argument probably does not 
 work without the assumption that this $t$-structure is countably smashing (in the naturally defined sense).

4. Assume once again that $w$ is $\alo$-smashing and left non-degenerate. Let $H:\cu \to \au$ be a homological functor that respects countable coproducts, where $\au$ is an abelian category; assume that $H$ kills $\cu_{w=i}$ for all $i>0$. 

Applying Proposition \ref{pbw}(\ref{iextcub}) we  obtain that $H$ also kills $\cu_{[1,i]}$ for all $i>0$. Thus Theorem \ref{tabo}(\ref{iabext},\ref{iabdhcl} ) implies that $H$ kills $\cu_{w\ge 1}$.

This observation can be used to obtain a nice easy proof of the stable Hurewicz theorem. 
\end{rema}


Now we describe those categories to which a countably smashing weight structure restricts.

\begin{pr}\label{prestrc}
Let $w$ be an $\al$-smashing (resp. smashing) weight structure structure on $\cu$ for $\al>\alz$. Then for an $\al$-localizing (resp. localizing) 
 subcategory $\du$ of $\cu$ the following conditions are equivalent.

A. $w$ restricts to $\du$.

B. $\obj \du=\obj \du' \star L $, where $L\subset \obj \cu_{+\infty}$ and $\du'$ is a triangulated subcategory of $\cu\abo$ such that $w$ restricts to it.

C. For any $M\in \obj \du$ we have $\linf(M)\in \obj \du$ (where $\linf:\cu\to \cu_{+\infty}$ is the functor defined in Theorem \ref{tabo}(\ref{iabdj})) and there exists  a choice of  $w_{\le i}M$ for all $i\in \z$ such that for the connecting morphisms $c_i:w_{\le i-1}M\to  w_{\le i}M$ (see Proposition \ref{pbw}(\ref{ifilt})) we have $\co(c_i)\in \obj \du$.

D.  There exists a class of objects $Q\subset \obj \du$ that $\al$-generates (resp. class-generates) $\du$ and the assumptions of condition C are fulfilled for its elements. 


\end{pr}
\begin{proof}
We recall that Theorem \ref{tabo}(\ref{iahcl}) says that $\rinf(M)\cong \hcl_i w_{\le i}M $ for $M\in \obj \cu$. 
Next, if $w$ restricts to $\du$ then for any $M\in \obj \du$ we can choose $w_{\le i}M\in \obj \du$. Since $\du$ is an $\alo$-localizing subcategory of $\cu$ we obtain that $\rinf(M)\in \obj \du$.
Now,  part \ref{iabdj} of the theorem 
 gives a distinguished triangle
 \begin{equation}\label{eldec} 
\rinf(M)\to M\to \linf(M)\to \rinf(M)[1]; 
\end{equation}
hence $\linf(M)$ belongs to $\obj \du$ as well. Moreover, 
the restriction of $w$ to $\du$ restricts to the triangulated subcategory $\du'$ whose object class equals $\obj \du\cap \obj \cu \abo$ according to Proposition \ref{prestr}(2). Thus condition A implies condition B.

Conversely, if condition B is fulfilled then for any $M\in \obj \du$ there exists a $\du$-distinguished triangle $X\to M\to Y \to X[1]$ with $X\in \obj \cu\abo $ and $Y\in \obj\cu_{+\infty}$ as well as a $w$-decomposition $\du$-morphism $w_{\le 0 }X\to X$. 
 The octahedron axiom of triangulated categories (along with Proposition \ref{pbw}(\ref{iext})) is easily seen to imply that the composition morphism 
$w_{\le 0}X\to M$ gives a weight decomposition for $M$.  Thus condition B implies condition A.

To prove that A implies C it remains to take (once again) for $M\in \obj \du$ a choice $w_{\le i}M\in \obj \du$ for all $i\in \z$  to obtain that  cones of the corresponding  $c_i$ belong to $\obj \du$ as well. 

Obviously, condition C implies condition D. 

Lastly assume that condition D is fulfilled. We argue somewhat similarly to the proof of Theorem \ref{tloc}(3). To verify condition A we should prove that 
the class $P=\cu_{w\le 0}\star (\obj \du \cap \cu_{w\ge 1})$ contains $\obj \du$. Proposition \ref{pstar}(1,2) 
  implies that $P$ is an $\al$-smashing (resp., smashing) extension-closed class of objects. Hence it suffices to verify that $P$ contains $\cup_{j\in \z}Q[j]$. Next, using the distinguished triangle (\ref{eldec}) we obtain that for any $M\in Q$ and $j\in \z$ it suffices to verify that $\rinf(M[j])\in P$. Next, the aforementioned isomorphism $\rinf(M)\cong \hcl_i w_{\le i}M \in \obj \du$ implies the following: it suffices to check that for any $i\in \z$ 
	there exists a choice of $w_{\le i}M$ belonging to  $ P[-j]\cap P[-1-j]$ (recall the definition of countable homotopy colimits). We take  choices of $w_{\le i}M$ coming from condition D and prove that they belong to $ P[-j]\cap P[-1-j]$ by induction on $i$. Certainly, this statement is fulfilled if $i<-j$. Next, our assumptions obviously imply that $\co(c_l)[m]	\in P$ for all $l,m\in \z$. Thus the distinguished triangle $\co(c_l)[-1]\to w_{\le l-1}M\to  w_{\le l}M\to \co(c_l)$ yields that $w_{\le l}M$ belongs to $P[-j]\cap P[-1-j]$ whenever $w_{\le l-1}M$ does. 

\end{proof}

\begin{rema}\label{restrc}

\begin{enumerate}
\item\label{irest1} One can certainly re-formulate the assumption that $w$ restricts to $\du'$ in condition B using the fact that the remaining conditions of our proposition are equivalent to each other. Note also that we have $\co(c_i)[-i]\in \cu_{w=0}$, and these objects are the terms of some choices of (weak) weight complexes both for $X$ and $\rinf(X)$ (see Proposition \ref{pwc} below). Thus the question whether $w$ restricts to $\du$ has an answer in terms of objects of $\cu_{+\infty}$ and of $\hw$. 

\item\label{irest2} Our proposition easily implies that  for any $\al$-smashing ($\cu,w$) there exists an inductive system of essentially small $\al$-localizing triangulated subcategories  $\cu_i\subset \cu$ such that $w$ restricts to them and $\cu$ is isomorphic 
to the $2$-colimit of $\cu_i$; cf. Proposition \ref{pwede}(4).

\end{enumerate}
\end{rema}

\subsection{Some examples and comments (to \S\ref{sloc})}\label{sex}


We start from an easy application of Theorem \ref{tloc}. We formulate it for the smashing setting;  however, 
part \ref{irl1} of the following proposition can be carried over to the $\al$-smashing setting without any problems.

\begin{pr}\label{prloc}
Let $w$ be a smashing weight structure on $\cu$, $B=B_1\cup B_2\subset \obj \cu$, where the elements of $B_1$ are left $w$-degenerate and the elements of $B_2$ are right $w$-degenerate; assume that the localization of $\cu$ by $\du=\lan B \rab$  exists (in the sense of Remark \ref{rclass}(3)) and denote it by $\pi: \cu\to \cu'$.

\begin{enumerate}
\item\label{irl1} Then $w$ descends to $\cu'$ and the restriction of $\pi$ to $\cu_b$ is a full embedding.


\item\label{irl3} Assume in addition that $\cu$ is generated by a  class 
 $Q$ of its compact objects as its own localizing subcategory,  $w$ is  class-generated by $\cu_{w=0}$, and  $Q\subset \obj \cu_b$.  Then not all elements of $\pi(Q)$ are compact in $\cu'$. Moreover, if 
 $Q$ and $B$ are sets then the descended weight structure $w'$ on $\cu'$  is perfectly generated (in the sense of  Remark \ref{rcgws}(\ref{iperf})).

\item\label{irl4} In particular, one can apply the previous assertions for $\cu=\shtop$, $w=w\sph$ (see Remark \ref{rcgws}(\ref{itop})), $B=B_2$ being any set of acyclic spectra (i.e., their singular homology is zero). 
 Moreover, $w'$ is class-generated by 
 $\{\pi(S)\}$ and  $\pi(S)$ is not compact. 
\end{enumerate}
\end{pr}
\begin{proof}
\ref{irl1}. 
$w$ descends to $\cu'$ according to Theorem \ref{tloc}(3(ii)). 

Now we prove the second part of the assertion.  Since $w$ is smashing, the objects of the localizing subcategory $\du_1=\lan B_1 \rab$ of $\cu$ 
are left weight-degenerate, and the objects of  
  $\du_2=\lan B_2\rab$ 
 are right weight-degenerate (see Proposition \ref{pbw}(\ref{icoprt},\ref{ismash}). 
 Next, Proposition \ref{pstar}(1,2) easily implies that  for any object  $D$ of $\du$ 
there exists a $\du$-distinguished triangle $$D_2\stackrel{f}{\to} D\stackrel{g}{\to} D_1\to D_2[1]$$ with 
$D_1\in \obj \du_1$ and $D_2\in \obj \du_2$.\footnote{One can easily prove that this triangle is functorially determined by $D$; yet we will not need this fact below.}  
If $D$  is also $w$-bounded then 
the orthogonality axiom for $w$ immediately gives the vanishing of $f$ and $g$ in this triangle; hence $D=0$. Thus the restriction of $\pi$ to $\cu_b$ is a full embedding according to Proposition \ref{pwede}(2).  

\ref{irl3}. Recall that the functor $\pi$ respects coproducts (see Theorem \ref{tloc}(3)). Since $\lan Q\ra\subset \cu_b$, the restriction of $\pi$ to this subcategory is a full embedding. Since $\pi$ is not a full embedding itself, (the easy)    Lemma 1.1.1(2) of \cite{bokum} implies that not all elements of $\pi(Q)$ are compact.

Next, if  $Q$ and $B$ are  sets then $\cu'$ is {\it well-generated} (see Theorem 7.2.1 of \cite{krauloc}). Since $w'$ is a smashing weight structure on $\cu$ (see Theorem \ref{tloc}(3)), it is perfectly generated according to Theorem 4.4.3(III.2) of \cite{bpure}.

\ref{irl4}. The only fact that remains to be  verified so that we can apply the previous assertions is that acyclic spectra in $\shtop$ are right $w\sph$-degenerate; this statement is 
 given by Theorem 4.2.3(2,6) of \cite{bkwn}. 

Next recall that $w\sph$ is class-generated by $\{S\}$ and $w'$ is smashing (see Theorem \ref{tloc}(3)); hence $w'$ is class-generated by   $\{\pi(S)\}$ according to Theorem \ref{tvneg}(\ref{itvn5}).
 We also obtain that not all elements of $\pi (\obj (\lan S\ra)) =\pi(\obj \shfin)$ are compact; it certainly follows that $\pi(S)$ is not compact. 
\end{proof}

We make some comments to this proposition along with Theorem \ref{tloc} itself.

\begin{rema}\label{rprloc}
1. Recall also that non-zero acyclic object in $\shtop$ do exist; see Theorem 16.17 of \cite{marg}. Moreover, under the assumptions of part \ref{irl3} of our proposition the weight structure $w'$ is also {\it (strongly) well-generated} in the sense of \cite[Remark 4.4.4(1)]{bpure}. Thus part \ref{irl4} of the proposition is a source of interesting 
 class-generated weight structures; they are (perfectly and) well-generated. Moreover, the authors suspect that $\cu'$ does not contain non-zero compact objects in this case; this would certainly imply that this $w'$  is not compactly generated in the sense of Remark \ref{rcgws}(\ref{icomp}).

2. Since $w$ is class-generated (under the assumptions of part \ref{irl3} of our proposition), $w$ is  left non-degenerate (see Theorem \ref{tvneg}(\ref{itvn2})); thus   $B_1$ is zero automatically.   Applying Proposition \ref{pwede}(2) we obtain that  the restriction of  $\pi$ to $\cu_+$ is a full embedding as well. We obtain that it suffices to assume that $Q\subset \obj \cu_+$ in this part of the assertion. 

On the other hand, for any compactly (or well) generated $\cu$\footnote{Recall that $\cu$ is said to be compactly generated whenever it is generated by a set of its compact objects as its own localizing subcategory.} 
 and its localizing subcategory $\du$ generated by a set $B$ 
  the quotient $\cu'$ is well-generated; thus if a smashing weight structure $w$ on $\cu$  descends to a weight structure $w'$ on $\cu'$  then $w'$ is strongly well-generated (and so, perfectly generated). In particular (to apply Theorem \ref{tloc}(3(ii))) one can take   $B=B_0\cup B_2$ with any sets $B_0\subset \cu_{[0,1]}$ and 
 $B_2$ consisting of  right $w$-degenerate objects. 


3. Arguing as in the proof of 
 Theorem \ref{tloc}(1) one can easily obtain the following statement. 

Let $F_0:\cu\to \cu_0$ and $F':\cu_0\to \cu'$ be exact functors such that the composition $F=F'\circ F_0$ is weight-exact (with respect to certain $w$ and $w'$), and assume in addition that $F_0$ is a localization functor ($\cu\to\cu/\du$) 
 and that for any $M,N\in \obj \cu$ the restriction of the homomorphism $\cu_0(F_0(M),F_0(N))\to \cu'(F(M),F(N))$ induced by $F'$ to the image of 
$\cu(M, N)$ in $\cu_0(F_0(M),F_0(N))$ is injective (in particular, this is certainly the case if $F'$ is injective on morphisms). Then for any $N\in \obj \du$ there exists a distinguished triangle $D(N)$ (see (\ref{ewwd})) in $\du$; thus $w$ descends to $\cu_0$.
\end{rema}

We also give 
  some examples of weight-exact localizations such that $\hpi$ is not an equivalence, and yet $\hw'$ has an easy description.

\begin{rema}\label{rrheart}
Adopt the assumptions and the notation of Theorem \ref{theart} and assume that  $B_0\subset \cu_{w=0}$ (this corresponds to $S_0$ of the form $\{0\to M\}$ for $M$ running through a subclass of $\cu_{w=0}$). We 
 describe the category $\hw[S\ob]\subset \hw'$ in the $\al$-smashing case; the corresponding statement in the smashing setting is similar.

In the $\al$-smashing case of the theorem all coproducts of less than $\al$ elements of $B_0$ are killed by 
 $\pi$; denote this class of objects by $B_0^{<\al}$. Then the additive functor from $\hw$ that is bijective on objects, surjective on morphisms, and kills all morphisms that factor through $B_0^{<\al}$ is easily seen to be equal to the localization $\hw\to \hw[S'{}\ob]$ (cf. the proof of Theorem \ref{theart}(4)). Hence  the category $\hw[S'{}\ob]\cong \hw[S\ob]$ 
 is isomorphic to the corresponding "additive quotient category" that may be denoted by $\hw/B_0^{<\al}$. 

This statement is closely related to a similar calculation given by Proposition 8.1.1(2) of \cite{bws}. It gives an easy (and rather "stupid") way of obtaining $\al$-smashing weight structures that are not smashing. One can localize a (smashing) category $\cu$ endowed with a smashing weight structure $w$ by $\lan B_0\ral$ for $B_0$ being (say) a subset of $\cu_{w=0}$. 
 Then the category $\kar(B_0^{<\al})$ is essentially 
small; hence it is not smashing and the localization functor $\pi$ cannot respect coproducts since its restriction $\hpi$ does not. 

Moreover, the authors suspect that neither $\cu'$ nor $\hw'$ can be smashing in this case. 
A useful observation here is that for $M_i\in \cu_{w=0}$ the coproduct of $\pi(M_i)$ (if exists) is a retract of $\pi(\coprod M_i)$.
We use it to demonstrate that $\hw'$ is not smashing in the case $\cu=\shtop$, $w=w\sph$, and $B_0=\{S\}$.  
Then  $\hw$ is isomorphic to the category of free abelian groups (so we will write its objects as abelian groups; recall that $S$ corresponds to $\z$), and $\hpi$ kills those morphisms whose images have less than $\al$ generators.  
We take all $M_i$ equal to $ \z\otal$ (the direct sum of $\al$ copies of $\z$) for $i$ running through $I=
 \al$.  If $\coprod_i\pi(M_i)$ exists then the corresponding idempotent endomorphism of $\pi(\coprod M_i)$ is easily seen to lift to an idempotent endomorphism of the group $\coprod M_i$. Hence $\coprod M_i\cong C\bigoplus D$ and the corresponding  morphisms $\pi(M_i)\to \pi(C)$ give $\pi(C)\cong \coprod \pi(M_i)$. Thus for the composition homomorphisms $pr^D_i:M_i\to D$ we have $\pi(pr^D_i)=0$; hence the images of these homomorphisms have less than $\al$ generators and their kernels are "rather large". Thus we can assume that $M_i\cong 
\z\bigoplus M_i'$ 
and 
 all of these copies of $\z$ belong to $C$. Therefore the 
 corresponding coproduct morphism $f: C\to \z^\oi$ (this is the direct sum of copies of $\z$ indexed by $I$) is surjective, whereas the compositions $f_i: M_i\to C\to \z^\oi$ have one-dimensional images. Thus $\pi(f_i)=0$ and $\pi(f)\neq 0$; hence these morphisms $M_i\to C$ do not actually give 
$\pi(C) \cong \coprod \pi(M_i)$.

\end{rema}

We also make some comments to Proposition \ref{pdesc}.

\begin{rema}\label{rrdesc}
1. 
Describing weight-exact functors appears to be an important problem. 

The only existing general recipe of constructing weight-exact functors that are surjective of objects is the theory of weight-exact localizations (see Theorem \ref{tloc}). However, some interesting examples where $w$ descends to $\cu$ and $F$ is not a localization functor can be easily constructed.
 According to Theorem 4.2.1(1) of \cite{bwcp} the aforementioned spherical weight structure $w_{\sph}$ restricts to the category $\shfin$ of finite spectra.  
We denote the resulting weight structure by   $w_{\sph}^{\fin}$; its heart consists of finite coproducts of $S$. 
 We take $F$ to be the singular homology functor whose target is $\cu'=D^b(\z)$. The latter category is equivalent to $K^b(\ffab)$, where $\ffab$ is the category of finitely generated free abelian groups. Hence $F$ is weight-exact with respect to $(w_{\sph}^{\fin}, \wstu)$ (see Remark \ref{rstws}(1)).
The functor $F$ is essentially surjective on objects since any object of $D^b(\z)$ splits as the (finite) direct sum of its shifted (co)homology (considered as finitely generated abelian groups) that can be lifted to (the corresponding shifts of) their Moore spectra (that are certainly finite). Moreover, $F$ is not a localization since it is conservative. 

This example can be generalized as follows: let $w$ be a bounded weight structure on $\cu$ such that $\hw$ is the category of finitely generated left projective  modules over a 
 left Noetherian left semi-hereditary ring $R$ (i.e., $R$ is an unital associative Noetherian ring such that any submodule of a left finitely generated projective $R$-module is projective). Assume moreover that there exists an exact  {\it strong weight complex} functor $F:\cu\to K^b(\hw)$   (see Conjecture 3.3.3 and  \S6.3 of \cite{bws} and  
Corollary 3.5 of \cite{sosnwc} 
 for  the definition and plenty of examples); 
  recall that this functor is essentially identical on $\hw$. Since any finitely generated left $R$-module has a two-term resolution by objects of $\hw$,   $F$ is essentially surjective on objects. Once again, $F$ is conservative immediately from Proposition 1.3.4(8) of \cite{bwcp}; hence $F$ is not a localization.

More generally, one can certainly combine functors constructed this way with weight-exact localizations.

2. Applying  Theorem 2.2.2(I.1) of \cite{bonspkar}  we obtain that  all parts of Proposition \ref{pdesc} 
can be generalized as follows: instead of an exact functor $F$ that is essentially surjective on objects one can take its composition with a full embedding $i:\cu'\to \cu''$ such that any object of $\cu''$ is a retract of an object of $i(\cu')$.\footnote{Actually, one has to generalize loc. cit. a little to obtain the corresponding version of Proposition \ref{pdesc}(2) for $(m,n)\neq (0,0)$; yet the 
 argument used in the proof carries over to this generality without any difficulty.} So, one may say that the corresponding weight structure on $\cu''$ (if exists) is descended from $w$ as well. 

\end{rema}

\subsection{On weight complexes and 
  subcategories}\label{swc}

We recall some basics on the theory of weight complexes. It was developed in \S3 of \cite{bws}; recall also that  in 
 \S1.3 of \cite{bwcp} some parts of the theory were exposed more carefully. In particular, in the latter paper it was noted that the weight complex functor is defined canonically only on a certain category $\cu_w$ that is (canonically) equivalent to $\cu$ and not on $\cu$ itself. 
Still distinctions between equivalent triangulated categories are irrelevant for our purposes; so make a choice of functor $t$  that is equivalent to the "canonic" weight complex $\cu_w\to K_\w(\hw)$. 
Here for any additive category $\bu$ the category  {\it weak homotopy category} $K_\w(\bu)$ of $\bu$-complexes is defined as follows.

\begin{defi}\label{dkw}
 $\obj K_\w(\bu)=\obj K(\bu)$, and 
$K_\w(\bu)(X,Y)=K_\w(\bu)(X,Y)/\backsim$ for complexes $X=(X^i)$ and $Y=(Y^i)\in \obj K_\w(\bu)$, where 
$\backsim$ denotes the following  ({\it weak homotopy}) relation: $f_1=(f_1^i):X\to Y$ is  weakly homotopic to $ (f_2^i)$ whenever  $f_1^i-f^i_2=d^{i-1}_Y\circ h^i+j^{i+1}\circ d^i_X$ for some collections of arrows $j^*,h^*:X^*\to X^{*-1}$ and all $i\in \z$. 
\end{defi}

\begin{pr}\label{pwc}
Let $\bu$ be an additive category; let $w$ be a 
weight structure on 
 $\cu$. 

Then the following statements are valid.

\begin{enumerate}
\item\label{iwc0} $K_\w(\bu)$ is a category, i.e., the weak homotopy equivalence relation is respected by compositions of $K(\bu)$-morphisms.

\item\label{iwcons} The natural projection $K(\bu)\to K_\w(\bu)$ is conservative, i.e., if  a  $K(\bu)$-morphism $h$ is an  isomorphism in  $K_\w(\bu)$ then $h$ is an isomorphism in  $K(\bu)$ as well.

\item\label{iwcf}
 An additive {\it weak weight complex} functor $t:\cu\to K_\w (\hw)$ is defined; 
$t$  respects coproducts of less than $\al$ objects if $w$ is $\al$-smashing (for a regular infinite cardinal $\al$). 

\item\label{iwcalc}
The restriction of $t$ to $\hw$ 
  embeds this category into $ K_\w(\hw)$.\footnote{Being more precise, 
 there exists an additive lift of $\hw$ to the aforementioned category $\cu_w$ such that the "canonical" weight complex functor $\cuw\to K_\w(\hw)$ 
 restricts to this embedding $\hw\to K_\w(\hw)$.}

\item\label{iwct}  If $
X[-1]\to Y\stackrel{f}{\to} Z\to X$ is a distinguished triangle in
$\cu$ then there exists a  lift 
 of $t(f)$ to a morphism $t'(f)\in K(\au)(t(Y),t(Z))$ such that  $t(X)\cong \co(t'(f))$. 

\item\label{iwcfwe}
If $F:\cu\to \du$ is a weight-exact functor (for some weight structure $v$ on a triangulated category $\du$) then the composition $t_v\circ F$ is isomorphic to $F_\w\circ t$, where $F_\w$ is the corresponding functor $K_\w(\hw)\to K_\w(\underline{Hv})$.

\item\label{iwcern}
$t(M)=0$ if and only if $M$ is a retract of an extension of a left degenerate object by a right degenerate one; if $\al>\alz$ then  these conditions are equivalent to $M$ itself being an extension of this form. Moreover, if $M$ is bounded below then $t(M)=0$ if and only if $M$ is left degenerate.
\end{enumerate}
 \end{pr}
\begin{proof}
Assertion \ref{iwc0} is immediate from Lemma 3.1.4(I.1) of \cite{bws}.

Assertion \ref{iwcons} is precisely Proposition 3.1.8(1) of ibid. 

\ref{iwcf}. The existence of $t$ is  given by Theorem 3.2.2(II) of \cite{bws} (see also Proposition 1.3.4(3,6) of \cite{bwcp}; cf. Remark 1.3.3(1) of \cite{bkwn})
 $t$ is additive by construction. If $w$ is $\al$-smashing then $t$   respects coproducts of less than $\al$ objects according to  
 Proposition 2.3.4(2) of \cite{bwcp} (cf. Proposition 2.3.2(5) of ibid.).
 
Assertion \ref{iwcalc} is an easy consequence of definitions; recall that the terms of $t(M)$ are essentially computed by means of Proposition \ref{pbw}(\ref{ifilt}); cf. Remark  1.3.5(5) of \cite{bwcp}.

Assertions \ref{iwct} and \ref{iwcfwe}   
 are given by  Proposition 1.3.4(7,10) of ibid.

The "main" part of  assertion \ref{iwcern} is given by  Theorem 3.1.6(I) of \cite{bkwn}. Next, if  $\al>\alz$ then $\cu$ is Karoubian; hence one can 
apply  Theorem 2.3.4(II.1) of ibid.  
 To obtain the "moreover" part one can note that the restriction of $w$ to $\cu_+$ is right non-degenerate (immediately from the orthogonality axiom); hence we can apply the dual to Proposition 3.1.8(1) of ibid. to it. 
\end{proof}

Now we explain how to combine our results with 
possible applications of weight complexes (that are closely related to  Remark 3.1.9 of \cite{bkwn} and Proposition 
1.9 of \cite{binters}).

\begin{rema}\label{rdescr}
We restrict ourselves to the $\al$-smashing setting; yet the smashing version of our observations are certainly fine also.
So, assume that $\cu$ is 
 endowed with an $\al$-smashing weight structure $w$.

1. Then for any $\cp\in \obj \cu$ we certainly have $t(\obj \lan \cp \ral) \subset \obj \lan t(\cp) \ra^{\al}_{K(\hw)}$ (see Proposition \ref{pwc}(\ref{iwcf},\ref{iwct}). It is an interesting problem to find conditions  ensuring that $M\in \obj \lan \cp \ral$ if $t(M)\in  \obj \lan t(\cp) \ra^{\al}_{K(\hw)}$.

2. Let $F:\cu\to \du$ be a weight exact-functor, where $\du$ is endowed with a left non-degenerate  weight structure $v$. Then for a $w$-bounded below $M\in \obj \cu$ we obtain that $F(M)=0$ if and only if  $K(\hf)(t(M))=0$, where $K(\hf): K(\hw)\to K(\hv)$ is the functor coming from the restriction $\hf$ of $F$ to hearts (see Definition \ref{dwso}(\ref{id5})).

Recall now that 
if $\cu$ and $\du$  are $\alo$-smashing and $F$ respects countable coproducts 
then $F$ sends almost bounded above objects of $\cu$ into that of $\du$,  and   $ \obj \du\abo\perp \obj \du_{+\infty}$  (see Remark \ref{rabo}(2)). 
 Hence  $F(M)$ is right weight-degenerate if and only if $K(\hf)(t(M))=0$; in particular, 
  for  $M\in \obj \cu\abo\cap \obj\cu_+$ we have   $K(\hf)(t(M))=0$ if and only if $F(M)=0$.

3. In the latter statement one can certainly take $F$ to be a weight-exact localization. So assume that $\al>\alz$, $\cu$ is $\al$-smashing, and $F=\pi$ is the localization by $\du=\lan \co(S_0)\ral$ for some 
 $S_0\subset \mo(\hw)$. Then for $M$ as above we obtain that $M\in \obj \du$ if and only if $K(\hf)(t(M))=0$; the smashing version of this statement is also valid.\footnote{If one 
 wants to avoid 
 the existence of localizations problem here then she can apply the corresponding  argument from the proof of Theorem \ref{theart}.} Moreover, Theorem \ref{theart}(3,4) (see also Remark \ref{rheart}(\ref{irh2}) provides us  with certain descriptions of the functor 
$\hf$ (that is "very simple" if $\du=\lan B_0\rab$ for some $B_0\subset \cu_{w=0}$; see Remark \ref{rrheart}) .
  
4. Now we would like to describe certain conditions that ensure for an object $M$ as above that it belongs to $\obj \du$ if it belongs to $\obj \cu_1\cap \obj \cu_2$ for
certain strict triangulated subcategories $\cu_i$ of $ \cu$; one may say that $\du$ is a   "candidate" for $\obj \cu_1\cap \obj \cu_2$.
If $\obj \du\subset \obj \cu_1\cap \obj \cu_2$ then this is equivalent to asking whether $\pi(M)=0$ whenever $M$ essentially belongs both to $\pi(\obj \cu_1)$ and to $\pi(\obj \cu_2)$. 

In \cite{binters} an interesting motivic example of this sort was considered; in it (as well as in the 
abstract Proposition 
1.9 of ibid. that was used for the proof) the vanishing of $t_{\cu'}(M)$ 
 essentially followed from $t_{\cu'}(\pi(\cu_1))\perp_{K(\hw')} t_{\cu'}(\pi(\cu_2))$. So we describe certain assumptions ensuring that the latter condition is fulfilled (even though it is certainly stronger than the weight-degeneracy of all elements of $\pi(\cu_1)\cap \pi(\cu_2)$). 
Our assumptions are much weaker than the ones considered in ibid.; we recall that (in the notation of Proposition \ref{pinters} below and 
 in addition to its assumptions) it was assumed that $w$ is class-generated by a set $\cp$ of compact objects,  whereas $\cp_1,\cp_2$, and $B=\co(S)$ were subsets of $\cp$.  
 \end{rema}

We would need a certain notion that is somewhat relevant for  the construction of weight structures. 

\begin{rema}\label{rpperf}
We will say that  $\cp$ is {\it $\al$-pseudo-perfect} (resp.  {\it pseudo-perfect}; cf. Remark \ref{rcompar}(\ref{iperf})) if $\cp$ is negative and $\cp\perpp$ is $\al$-smashing (resp. smashing).

Obviously, if this is the case then $\cp$ is $\al$-negative (resp. class-negative); hence 
the pseudo-perfectness conditions can be useful for the application of Theorem \ref{tvneg} and Corollary \ref{classgws}, respectively. 
\end{rema}

\begin{pr}\label{pinters}
Let $w$ be  a 
smashing weight structure on (a smashing) $\cu$, $M\in \obj \cu_+\cap \obj \cu\abo$,  $S_0$ is a class of $\hw$-morphisms, and   
$\du=\lan \co(S_0)\rab$.
Assume that there exists a universal  functor $L:\hw\to \hu$ among those additive functors that respect coproducts and make all elements of $S_0$ invertible, 
 and that $\cp_1,\cp_2\subset \obj \cu$ satisfy one of  the following 
  conditions:

(1) the class  $K(L)(t(\cp_1))$ is pseudo-perfect in $K(\hu)$ 
  and $K(L)(t(\cp_1))\perp_{K(\hu)} (\cup_{i\in \z}K(L)(t(\cp_2)[i]) )$;

(2) $\cp_2$ is class-negative and for any $C\in \cp_1$ there exist $n_1,n_2\in \z$ such that $K(L)(t(C))\perp_{K(\hu)}K(L)(t(\cup_{n_1<i<n_2}\cp_2[i]\cup \cu_{2,w_2\le n_1}\cup \cu_{2,w_2\ge n_2})) $, where $w_2$ is the weight structure on $\cu_2= \lan \cp_2\rab$ that is class-generated by $\cp_2$.

Then $M$ belongs to $\obj \du$ whenever it belongs to $\obj \cu_1\cap \obj \cu_2$.

\end{pr}
\begin{proof}
Similarly to Theorem \ref{theart} (see its proof and Remark \ref{rheart}(\ref{irhcl})) 
 it suffices to prove the obvious $\al$-smashing  version of the proposition for all regular cardinals $\al>\alz$ 
 and all small $\cu$; this is what we will actually do (to avoid class-categories and larger 
 universes).\footnote{Certainly, in this setting the existence of $L$ is automatic; see Theorem \ref{theart}(3).} 
Next, as we have explained above, it suffices to verify that $t_{\cu'}(\pi(M))=0$, and the latter certainly follows from $t_{\cu'}(\pi(M))\perp t_{\cu'}(\pi(M))$. Thus it suffices to verify that $t_{\cu'}\circ \pi(\obj \cu_1)\perp_{K(\hw')} t_{\cu'}\circ \pi(\obj \cu_2)$. Once again, in the $\al$-smashing setting the latter condition follows from  $\lan t_{\cu'}\circ \pi(\cp_1)\ra^{\al}_{K(\hw')}\perp \lan t_{\cu'}\circ \pi(\cp_2)\ra^{\al}_{K(\hw')}$ (see Remark \ref{rdescr}(4)). 

Since $\hu$ is a full subcategory of $\hw'$ (see Theorem \ref{theart}(3)), the latter assumption follows from 
 (1) easily. So we assume that  condition (2) is fulfilled.

 Let us now fix some $C\in \cp_1$.
We recall that 
$\cu_{2,[n_1+1,n_2-1]}$  equals the extension-closure of $\cup_{n_1<j <n_2}\cu_{2,w_2=j}$ (see Proposition \ref{pbw}(\ref{iextcub})), 
whereas $\cu_{2,w_2=j}$ is the Karoubi-closure of the class of all 
$\{\coprod_{i\in I}P_i[j]\}$, where $P_i\in \cp_2$ and $I$ is of cardinality less than $\al$ (see Theorem \ref{tvneg}(\ref{itvn3})); 
 hence our assumptions imply that $t_{\cu'} (\pi(C))\perp_{K(\hw')} t_{\cu'}\circ \pi(\cu_{2,[n_1+1,n_2-1]})$ (here we apply Proposition \ref{pwc}(\ref{iwcf},\ref{iwct}). Since $t_{\cu'} (\pi(C))$ is also $K(\hw')$-orthogonal to $t_{\cu'}\circ \pi(\cu_{2,w_2\ge n_2})$, combining Proposition \ref{pbw}(\ref{iwd0}) with Proposition \ref{pwc}(\ref{iwct}) we also obtain  that $t_{\cu'} (\pi(C))\perp t_{\cu'}\circ \pi(\cu_{2,w_2\ge n_2})$. Recalling that $t_{\cu'} (\pi(C))$ is orthogonal to $t_{\cu'}\circ \pi(\cu_{2,w_2\le n_1})$ and 
considering $n_1$-weight decompositions of  objects of $\cu_2$ we obtain $t_{\cu'} (\pi(C))\perp t(\cu_2)$. Since this is fulfilled for every $C\in \cp_1$, applying Proposition \ref{pwc}(\ref{iwcf},\ref{iwct}) once again we conclude that  $t_{\cu'}\circ \pi(\obj \cu_1)\perp_{K(\hw')} t_{\cu'}\circ \pi(\obj \cu_2)$ as desired.
\end{proof}

\begin{rema}\label{rinters}
1. The existence for each $C$ of $n_1,n_2\in \z$ such that $K(L)(t(C))\perp_{K(\hu)}K(L)(t( \cu_{2,w_2\le n_1}\cup \cu_{2,w_2\ge n_2})) $
 immediately follows from the following conditions: for the composed functor $G:\cu_2\to \cu\to \cu'$ there exist $m_1,m_2\in \z$ such that $G\circ [m_1]$ is left weight-exact, $G\circ [m_2]$ is right weight-exact (with respect to $w_2$ and $w'$), and $\pi(C)$ is $w'$-bounded for each $C\in \cp_1$. Certainly, if one restricts himself to the case $\cp_2\subset \cu_{w=0}$ then $G$ is weight-exact (see Theorem \ref{tvneg}(\ref{itvn4}); thus one can take $m_1=m_2=0$).

2.  If one only assumes that $M\in\obj \cu\abo$ then the remaining assumptions would imply that $\pi(M)$ is right weight-degenerate if $M\in \obj \cu_1\cap \obj \cu_2$ (see Remark \ref{rdescr}). 
 
3. Instead of the two weight decomposition arguments used in the proof one may  apply an easy {\it weight spectral sequence} argument; see Theorem 2.3.2 of \cite{bws}.

\end{rema}

\subsection{Two  lemmas on localizations}\label{slloc}


Now we prove two statements that were applied above. Probably, both of them are well-known to the experts; the first proof follows an argument suggested by D.-Ch. Cisinski.

\begin{pr}\label{addloc}
Let $\al$ be a regular infinite cardinal; suppose $\bu$ is an $\al$-smashing additive category, $S$ is a set of morphisms in $\bu$ containing identities and closed under coproducts of less than $\al$ elements. Then if the localized category $\bu[S^{-1}]$ exists then it is additive and $\al$-smashing, and  the localization functor $\pi:\bu\to \bu[S^{-1}]$ is additive and respects coproducts of less than $\al$ objects.

\end{pr}
\begin{proof}
To check that the category is additive it suffices to verify that 
(binary)  coproducts and  products exist in it, $A \times B \cong A \coprod B$ for any pair of objects, 
and every object is a group object.

Firstly, note that every adjoint pair of functors $L : \bu \leftrightarrows \hu : R$
induces an adjoint pair $\hat{L} : \bu[S^{-1}] \leftrightarrows \hu[T^{-1}] :
\hat{R}$ if $L(S) \subset T$ and $R(T) \subset S$. Indeed, the latter assumptions imply
that the functors $Loc_T \circ L$ and $Loc_S \circ R$ send $S$ and $T$ into
isomorphisms of the categories $ \hu[T^{-1}]$ and $ \bu[S^{-1}]$, respectively; hence they induce some functors $\hat{L}$ and $\hat{R}$ on the
localized categories.
Now, the unit and the counit for the  pair $(L,R)$ yield the unit and the counit of the pair
$(\hat{L},\hat{R})$. 

Now we apply the above observation to the functor $\bu \to \operatorname{pt}$, where $\operatorname{pt}$ denotes the category with one object and one morphism and its right and left adjoints (i.e., terminal and initial objects). This yields that zero object exists in the category $\bu[S^{-1}]$. 

Similarly we apply this observation to the diagonal functor $\bu \to \bu \times \bu$ along with its right and left adjoints (i.e., with the binary product
and coproduct functors, respectively). 
Note that the natural functor $(\bu\times\bu)[(S \times S)^{-1}] \stackrel{e}\to \bu[S^{-1}]\times \bu[S^{-1}]$ is an equivalence of categories. Indeed, 
the correspondence $\bu[S^{-1}](A,B)\times \bu[S^{-1}](A',B') \to (\bu\times \bu)[(S\times S)^{-1}]((A,A'),(B,B'))$ that maps $(f,f') \in \bu[S^{-1}](A,B)\times \bu[S^{-1}](A',B')$ into the element $(f,\id_{B'})\circ (\id_A,f')$ defines an inverse functor to $e$ (it is well-defined since $S$ contains all identities). 
We obtain that in $\bu[S^{-1}]$
all 
binary coproducts and products exist, the localization functor preserves them, and the natural morphism $A \coprod B \to A \times B$ is an isomorphism for any objects $A$ and $B$.

Hence every object of $\bu[S^{-1}]$ possesses the induced structure of a monoid object. 
 Let us verify that this monoid structure is a group structure.
The localization functor $Loc_S$ preserves (binary) products; thus the image of
a group object has the structure of group object. 
 Since $Loc_S$ is
surjective on objects, we obtain the result.

Lastly,  for any index set $I$ of cardinality less than $\al$ the functor $\coprod_I:\bu^I\to \bu$ is  left adjoint to the diagonal functor $\bu \to \bu^I$. Hence the same adjunction argument as above yields that $\bu[S\ob]$ has coproducts and $\pi$ respects them. 
\end{proof}


\begin{pr}\label{plocsub}
Let $\du$ and $\eu$ be full strict triangulated subcategories of $\cu$. Assume that the Verdier localization functor $\pi:\cu\to\cu/\du$ exists and
that for any 
$D\in \obj \du$ and $E\in \obj \eu$ there exists a $\du$-distinguished triangle \begin{equation}\label{eww} D'\to D\to D''\to D'[1]
\end{equation} 
 such that $E \perp D''$ and $D'\in \obj \eu$. 

Then the restriction of $\pi$ to $\eu$ is isomorphic to the localization 
  of $\eu$ by the triangulated subcategory whose object class is $\obj \du\cap \obj \eu$ (and so, the latter localization exists in the sense of Remark \ref{rclass}(3)).
	
	In particular, this statement is valid if $\cu$ is endowed with a weight structure $w$, $\eu=\cu_-$, and for any object of  $D_0$ of $\du$ there exists a $\du$-distinguished triangle \begin{equation}\label{ewl} D_1\to D_0\to D_2\to D_1[1]\end{equation} 
	(\ref{eww}) such that $D_1\in  \obj \cu_-$ and $D_2\in \cu_{w\ge 0}$.

\end{pr}
\begin{proof}
According to Lemma 10.3.13 of \cite{weibook} (we use the equivalence (1)$\iff$(2) there), we should check that for any $\cu$-morphism $f:M \to E$ such that $E\in \obj \eu$ and $\co f\in \obj \du$ there exists a $\cu$-morphism $f':M'\to M$ such that $M'\in \obj \eu$ and $\co (f\circ f')\in \obj \du$.

Now we complete $f$ to a distinguished triangle $M\stackrel{f}{\to}E\stackrel{g}{\to}D\to M[1]$ and choose a $\du$-distinguished triangle (\ref{eww}) (for $D=\co(f)$). Since $E\perp D''$, the morphism $g$ factors through some $g':E\to D'$. Then  $M'=\co (g)[-1]\in \obj \eu$ and 
the octahedron axiom of triangulated categories gives the existence of $f'\in \cu(M',M)$ such that $\co (f\circ f')\in \obj \du$ (since $\co f'=D''[-1]\in \obj \du$).

It remains to verify the "in particular" part of our assertion. We fix $D\in \obj D$ and $E\in \obj \cu_-$ and verify the existence of a distinguished triangle (\ref{eww}) as desired.  Since $E\in \obj \cu_-$, $E$ belongs to $\cu_{w\le i}$ for some $i\in \z$. Now, 
we take $D_0=D[-i-1]$, choose some distinguished triangle (\ref{ewl}), and shift it by $i+1$ to obtain a candidate for the triangle (\ref{eww}). Then (\ref{eww}) is certainly a $\du$-distinguished triangle, $D'=D_1[i+1]\in \obj \cu_-$, and  since $D''=D_2[i+1]\in \cu_{w\ge i+1}$ we have $E\perp D''$ (by the ortogonality axiom for $w$). 


\end{proof}

\end{document}